%% file: paper.tex
\documentclass[11pt,reqno]{article}
\input{preamble.tex}

\title{A Model-Informed Deep Learning Algorithm for Solving the Phaseless Inverse Scattering Problem}
 \author{Dinh-Liem Nguyen\thanks{Department of Mathematics, Kansas State University, Manhattan, KS 66506 (\href{mailto:dlnguyen@ksu.edu}{dlnguyen@ksu.edu})} \and  Nhung H. Nguyen\thanks{Department of Mathematics, Kansas State University, Manhattan, KS 66506 (\href{mailto:nhungnh@ksu.edu}{nhungnh@ksu.edu})} \and  Aravinth Ravi\thanks{Department of Mathematics, Kansas State University, Manhattan, KS 66506  (\href{mailto:arav0006@ksu.edu}{arav0006@ksu.edu})}} 
\begin{document}
\date{}
\maketitle
\begin{abstract}


We study in this paper an unsupervised, two-step, model-informed deep learning framework for solving the phaseless inverse scattering problem. The objective is to reconstruct a compactly supported function that characterizes a scatterer from boundary measurements of the modulus of the total wave corresponding to multiple incident waves. In the first step, the phaseless data are transformed into an imaging function that is directly related to the underlying scatterer. Motivated by the contrast source method, we derive a coupled system of model equations using the Lippmann–Schwinger integral equation and a spectral-based equation involving the imaging function.
In the second step, the unknown scatterer and contrast source function are each parameterized as independent feedforward neural networks, which are trained simultaneously at each iteration using the derived system of model equations. The Lippmann-Schwinger integral equation enforces the underlying physics as a model constraint, while the spectral-based equation incorporating the imaging function supplies geometric information on the shape and location of the unknown scatterer. The resulting framework achieves accurate and efficient reconstructions while maintaining robustness to measurement noise. Furthermore, transforming some aspects of the problem to the spectral domain enables an effective reduction in the computational cost, allowing us to recover the unknown scatterer quickly. Numerical experiments in both two- and three-dimensional settings demonstrate the effectiveness, stability, and practical potential of the proposed approach.

\end{abstract}

\section{Introduction}

Consider an inhomogeneous medium that fills up a bounded Lipschitz domain $D \subset \R^{n}$ for $ n=2 \text{ or } 3$.  Suppose that the medium is characterized by bounded function $q$, satisfying $q = 0$  in $\R^{n} \backslash \overline{D}$. Consider the incident plane wave
\begin{align*}
    u^{in}(x,d) = e^{ikx \cdot d},\qquad x \in \R^{n},\, d \in \mathbb{S}^{n-1} := \{ x \in \R^{n}: |x| = 1 \},
\end{align*}
where $k > 0$ is the wave number and $d$ is the direction vector of propagation. The scattering of $u^{in}$ by the inhomogeneous medium is described by the following model problem:
\begin{align} \label{eq:PDE}
    & \Delta u + k^{2} (1 +q(x))u = 0, \hspace{1.9 cm} x \in \R^{n}, \\ \label{eq:totalfield}
    & u = u^{in} + u^{sc}, \\ \label{eq:radcon}
    & \lim_{r \rightarrow \infty} r^{\frac{n-1}{2}} \left( \frac{\partial u^{sc}}{\partial r} -iku^{sc} \right) = 0, \qquad r = |x|,
\end{align}
where $u$ is the total field, $u^{sc}$ is the scattered field, and the Sommerfeld radiation condition \eqref{eq:radcon} holds uniformly for all directions $x/|x| \in \mathbb{S}^{n-1}$. If $\mathbb{R}^{n} \backslash \bar{D}$ is connected and $\mathrm{Re}(q) \geq 0$,  this scattering problem is known to have a unique weak solution $u^{sc} \in H^{1}_{\text{loc}}(\R^n)$, see~\cite{Colton1992InverseAA}. Let $\Omega$ be a  ball of sufficiently large radius $R_\Omega$ with boundary $\partial \Omega$ such that $\bar{D} \subset \Omega$.  We are interested in the following inverse problem.

\vspace{0.2cm}
\noindent
\textbf{Inverse Problem.}  Determine $q$, given $|u|$ on $\partial\Omega$ associated with all incident directions $d \in \mathbb{S}^{n-1}$. 

\vspace{0.2cm}

Results involving the uniqueness and stability of solution to the inverse scattering problem with phaseless data has been established in \cite{ Xu2017UniquenessII, zhang2020unique} for multiple incident directions and in \cite{Klibanov2014PhaselessIS,Klibanov2015ReconstructionPF,Novikov2015FormulasFP} for multiple frequencies.
Inverse scattering problems are concerned with the recovery of geometric and physical properties of unknown scatterers from boundary measurement data. Such problems arise naturally in a wide range of scientific and engineering disciplines, including subsurface reservoir identification, medical imaging, geophysical prospecting, and microwave remote sensing \cite{Bulyshev2004ThreedimensionalVM,Persico2014IntroductionTG}. A variety of numerical methods have been developed over the past few decades to address inverse scattering problems with phased data \cite{Colton1992InverseAA,beilina2012approximate,cakoni2022inverse,ito2014inverse}.

However, acquiring phased data in practice is often difficult and expensive. In many applications, phaseless data are considerably easier and more cost-effective to obtain, which makes phaseless reconstruction particularly attractive. On the other hand, the absence of phase information renders the inverse problem  more nonlinear and severely ill-posed.
For numerical methods for inverse scattering problems with phaseless data, we refer to \cite{Bao2013NumericalSO, Dong2018ARB, Ji2019InverseAS, Chen2015PhaselessIB, Zhang2018FastIO, Zhang2020AnAF} for qualitative reconstruction results, including the recovery of the shape and location of the scatterer. In contrast, results on coefficient reconstruction are relatively limited. A widely used approach is to first recover the phase information from phaseless measurements and then apply existing numerical methods for phased inverse scattering problems to reconstruct the coefficient $q$; see \cite{klibanov2016two, klibanov2019coefficient, klibanov2018numerical, Ammari2013PartialDR, le2026inverse}.



Recent advances in machine learning, particularly deep learning, have pioneered new paths for addressing the inverse scattering problem \cite{Chen2020ARO,Rzio2026ARO}. Deep learning's capacity to capture nonlinearity and complexity from observed data, generalize from computationally inexpensive synthetic data to costly real-world measurements, and handle incomplete data through techniques such as dropout, regularization, and data augmentation makes it an attractive framework for algorithm design. Deep learning approaches can broadly be classified into two categories: physics-agnostic approaches \cite{Gao2021OnAA,Fan2019SolvingEI,Zhang2022SolvingTW,Zhou2022ANN,nguyen2024tnet}, in which the loss function is determined solely by classical error measures with optional regularization \cite{Khoo2018SwitchNetAN} and no explicit physical prior; and physics-informed approaches, in which knowledge of the underlying physics is explicitly incorporated into the loss function or embedded into the neural network architecture \cite{Chen2019PhysicsinformedNN, Pokkunuru2023ImprovedTO,Zhang2023SolvingAI,Nguyen2024TAENAM,Yin2024PhysicsawareDL,Guo2021PhysicsED} .

Despite this broader progress, corresponding advances in solving inverse problems with phaseless data have been comparatively limited, as reflected in the sparse available literature. The papers \cite{Gao2019MachineLB} and \cite{Yin2020ANN} propose two-step supervised deep learning methods for the inverse obstacle problem, while \cite{ning2025direct} develops a two-step supervised deep learning approach for the inverse scattering problem. All of these methods rely on supervised learning frameworks that require large, labeled datasets for training. As a result, their performance is strongly tied to the distribution of the training data: while they often achieve high-resolution reconstructions when the true parameters lie within this distribution, performance may degrade substantially for out-of-distribution inputs.

Hence, in this paper, we propose a novel unsupervised Fourier-accelerated method which only requires the phaseless boundary measurement data.  Our approach integrates a system of approximate equations derived from the Lippmann-Schwinger integral equation and an equation relating the imaging function to the contrast, with a model-based deep learning framework and is inspired by the  physics-informed deep learning method \cite{Raissi2019PhysicsinformedNN,Cuomo2022ScientificML}. The imaging function utilizes the phaseless boundary measurement data to characterize the unknown scatterer and provide a stable geometric estimate of its location and shape. Furthermore, its approximation can be reformulated as an equation relating the imaging function data to the scatterer, which forms the second equation of the model system. We exploit the convolutional structure of the two above mentioned equations and apply the Fourier transform to derive the state-data equations, which form the objective function in the model-based deep learning algorithm. This derivation of the objective function has been inspired by the contrast source method  \cite{Berg2001ContrastSI}. 
Next, using this system of approximate model equations as an objective function, we concurrently optimize two neural networks to approximate the Fourier coefficients of the contrast source function and the unknown scatterer. 


A related framework was previously developed by the authors for the acoustic inverse source problem \cite{submitted}, where an imaging function capturing the location and shape of the unknown source was integrated in a model equation to recover the Fourier coefficients of the source function. The present work extends this idea to the significantly more challenging setting of phaseless inverse scattering. In contrast to the linear inverse source problem, the lack of phase information in phaseless inverse scattering greatly amplifies both the nonlinearity and ill-posedness of the reconstruction problem.
To overcome these additional challenges, we introduce a new framework based on a system of approximate model equations and jointly train two independent neural networks through a coupled bi-objective loss function, rather than relying on a single network as in the previous work.


This algorithm achieves a stable and rapid reconstruction of the scatterer while requiring the neural network to approximate the contrast source function using far fewer inputs, rather than estimating the parameter over the entire computational domain. Moreover, applying the Fourier transform converts the convolutional structure of the volume integral into pointwise multiplication of Fourier coefficients, drastically reducing computational cost and avoiding the singularity of the Green's function in the state equation. Furthermore, since our method does not rely on an input-output training dataset, the algorithm is free from the constraints imposed by such datasets.

The paper is organized as follows. The imaging function and its analysis are presented in Section \ref{section2}. Section \ref{section3} is dedicated to derivation of the Fourier-accelerated system of approximate model equations. Section \ref{section4} describes the model-based neural network. The numerical study for simulated data is presented in Section \ref{section5}.

\section{A stable imaging function}\label{section2}

In this section, we present an imaging function designed to efficiently process phaseless data measured on a sufficiently distant boundary. In addition to furnishing an equation within the model-informed system, which is discussed in Section \ref{section3}, this imaging function provides an initial estimate of the geometric characteristics of the unknown contrast $q$.


It is well known that the solution of \eqref{eq:PDE}-\eqref{eq:radcon} can be characterized equivalently as the solution of the Lippmann–Schwinger integral equation 
\begin{align}\label{eqn:lippmannschwringerqn}
    u^{sc}(x,d) = k^{2} \int_{\Omega} \Phi(x,y) q(y) u(y,d) dy,\qquad x \in \R^{n}, \, d\in \mathbb{S}^{n-1},
\end{align}
where $\Phi$ is the Green's function of the scattering problem \eqref{eq:PDE}-\eqref{eq:radcon} which is given by
\[
\Phi(x,y)=\Phi(x-y) =
\begin{cases}
\displaystyle 
\frac{i}{4} H_0^{(1)}\!\bigl(k\lvert x - y\rvert\bigr), & n = 2, \\[8pt]
\displaystyle 
\frac{\exp\bigl(ik\lvert x - y\rvert\bigr)}{4\pi \lvert x - y\rvert}, & n = 3.
\end{cases}
\]
for $x, y \in \R^{n}$ and $x\neq y$.
Following the approach studied in~\cite{Chen2015PhaselessIB} we define $u_p$ as follows
\begin{align}\label{phaselessformula}
    u_p(x,d) = \frac{|u(x,d)|^{2}-|u^{in}(x,d)|^{2}}{\overline{u^{in}(x,d)}}. 
\end{align}
Inspired by~\cite{Le2022SamplingTM}, for $z\in \mathbb{R}^n$, we define the imaging function with phaseless data as follows
\begin{align*}
    \mathcal{I}(z,d) := \int_{\partial \Omega} \Psi(x,z)u_p(x,d)    dS(x),
\end{align*}
with
$$\Psi(x,z):=\dfrac{\partial \operatorname{Im}\Phi(x,z)}{\partial\nu(x)}  - ik \operatorname{Im}\Phi(x,z),$$
where $\nu$ is the outward normal unit vector of $\partial \Omega$  and    $$\operatorname{Im}\Phi(x,z) = \begin{cases}
        \frac{1}{4}J_0(k|x-z|),& n=2, \\ \notag
        \frac{k}{4\pi} j_0(k|x-z|),&  n=3.
    \end{cases}$$
  The test function $\Psi$ was introduced in \cite{Le2022SamplingTM} for constructing a stable imaging functional in inverse scattering problems with phased data. A key advantage of $\Psi$ over test functions commonly used in direct sampling methods is that it does not have a singularity.

As with other imaging functions in sampling-based approaches, we will partially justify in the following that the imaging function 
$\mathcal{I}(z,d)$  attains small values when the sampling point $z$ lies outside the scatterer, and significantly larger values when $z$ is located inside the scatterer. Consequently, by evaluating and visualizing $\mathcal{I}(z,d)$ over a sampling domain, one can reconstruct the location and shape of the scatterer. 

With the observation that
\begin{align*}
    u_p(x,d) & = u^{sc}(x,d) +\frac{|u^{sc}(x,d)|^{2}}{\overline{u^{in}(x,d)}} + \frac{\overline{u^{sc}(x,d)}u^{in}(x,d)}{\overline{u^{in}(x,d)}},
\end{align*}
we can expand $\mathcal{I}$ as
\begin{align}
  \notag  
   \mathcal{I}(z,d)= & \int_{\partial \Omega} \Psi(x,z) u^{sc}(x,d)    dS(x)  \\ \label{Iz_phaseless}
   & \hspace{1 cm}+ \int_{\partial \Omega} \Psi(x,z) \frac{|u^{sc}(x,d)|^{2}}{\overline{u^{in}(x,d)}}   dS(x) + \int_{\partial \Omega} \Psi(x,z)\frac{\overline{u^{sc}(x,d)}u^{in}(x,d)}{\overline{u^{in}(x,d)}}  dS(x).
\end{align}
Now let us define
$$\widetilde{\mathcal{I}}(z,d) :=  \int_{\partial \Omega} \Psi(x,z) u^{sc}(x,d)    dS(x).$$
The next theorem shows that the last two integrals in \eqref{Iz_phaseless} become negligible as $R_\Omega \gg1$, where $R_\Omega$ is the radius of the measurement boundary $\partial \Omega$. Consequently, for sufficiently large 
$R_\Omega$, we have $ \mathcal{I} \approx \widetilde{\mathcal{I}}$.

\begin{theorem}
\label{thm:Izdecay}
     For any $z \in \Omega$, we have
     \begin{align}\label{Iz_asym}
    \mathcal{I}(z,d) =\widetilde{ \mathcal{I}}(z,d) + \mathcal{O}\left(R_\Omega^{(1-n)/2}\right). 
\end{align}
\end{theorem}
\begin{proof}
    Following a similar argument used in \cite{ning2025direct, Chen2015PhaselessIB}, we can prove that
     \begin{align*}
        \left| \int_{\partial \Omega} \Psi(x,z)\frac{|u^{sc}(x,d)|^{2}}{\overline{u^{in}(x,d)}}   dS(x) \right| \le CR_\Omega^{(1-n)/2},
    \end{align*}
   and
    \begin{align*}
        \left| \int_{\partial \Omega} \Psi(x,z)\frac{\overline{u^{sc}(x,d)}u^{in}(x,d)}{\overline{u^{in}(x,d)}}  dS(x) \right|
        \leq C R_\Omega^{(1-n)/2}.
    \end{align*}
     where $C>0$. These results directly imply \eqref{Iz_asym}.
\end{proof}


The preceding theorem establishes that, for sufficiently large $R_\Omega$, the imaging function $\mathcal{I}$ admits the approximation $\widetilde{\mathcal{I}}$. It is therefore essential to examine the behavior of $\widetilde{\mathcal{I}}$. Since the measurement data are collected on $\partial \Omega$ with sufficiently large radius $R_\Omega$, the Sommerfeld radiation condition \eqref{eq:radcon} implies that $\partial u/\partial \nu \approx ik u$. Consequently, the imaging function can equivalently be expressed as follows
\begin{align}\label{thm:imaging}
    \widetilde{\mathcal{I}}(z,d) = \int_{\partial \Omega} \left[\frac{\partial \mathrm{Im}\Phi(x,z)}{\partial \nu(x)} u^{sc}(x) -\mathrm{Im}\Phi(x,z)\frac{\partial u^{sc}(x)}{\partial \nu(x)} \right] dS(x). 
\end{align}
As studied in~\cite{Le2022SamplingTM}, $\widetilde{\mathcal{I}}(z,d)$ satisfies
  \begin{align*}
       \widetilde{\mathcal{I}}(z,d) = k^{2} \int_{D} \mathrm{Im}\Phi(z,x) q(x)u(x) dx.
  \end{align*}
    This result helps us understand the resolution of the imaging function $\widetilde{\mathcal{I}}(z,d)$ through the localization properties of $J_0(k|y - z|)$ and $j_0(k|y - z|)$. Specifically, these functions exhibit a  peak at $z = y$ and decay rapidly as $|y - z|$ increases. Together with Theorem \ref{thm:Izdecay}, we can partly justify the desired behavior of  ${\mathcal{I}}(z,d)$ as an imaging performance by
    \begin{equation}
        \label{approxI}
        I(z,d) =  k^{2} \int_{D} \mathrm{Im}\Phi(z,x) q(x)u(x) dx+   \mathcal{O}\left(R_\Omega^{(1-n)/2}\right).\end{equation}
        In Figure~\ref{figI(z)} we can see that the summation of $\mathcal{I}(z,d)$ over multiple directions of incident waves $d$ provides a reasonable estimate for the location and shape of the scatterer.

\begin{figure}[H]
    \centering
    \begin{subfigure}[b]{0.4\textwidth}
         \centering
        \includegraphics[width=0.9\textwidth]{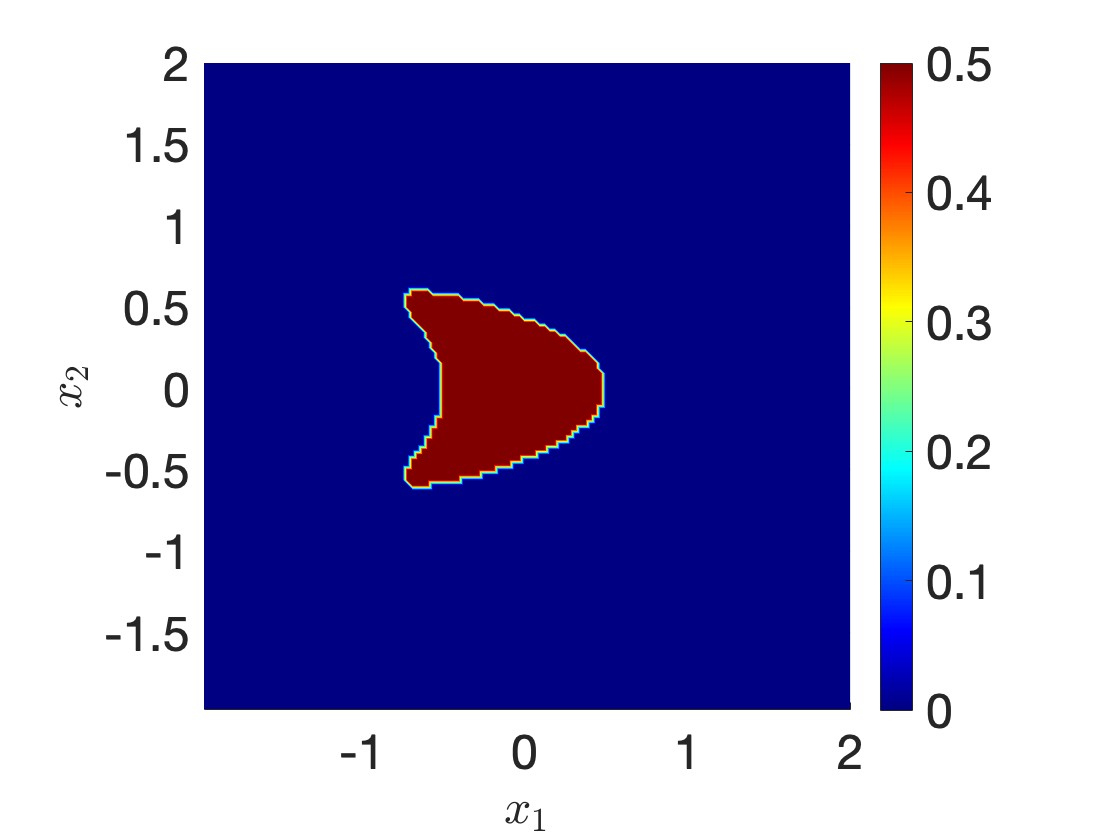}
        \caption{True $q$}
        
    \end{subfigure}
    \begin{subfigure}[b]{0.4\textwidth}
             \centering
         \includegraphics[width=0.9\textwidth]{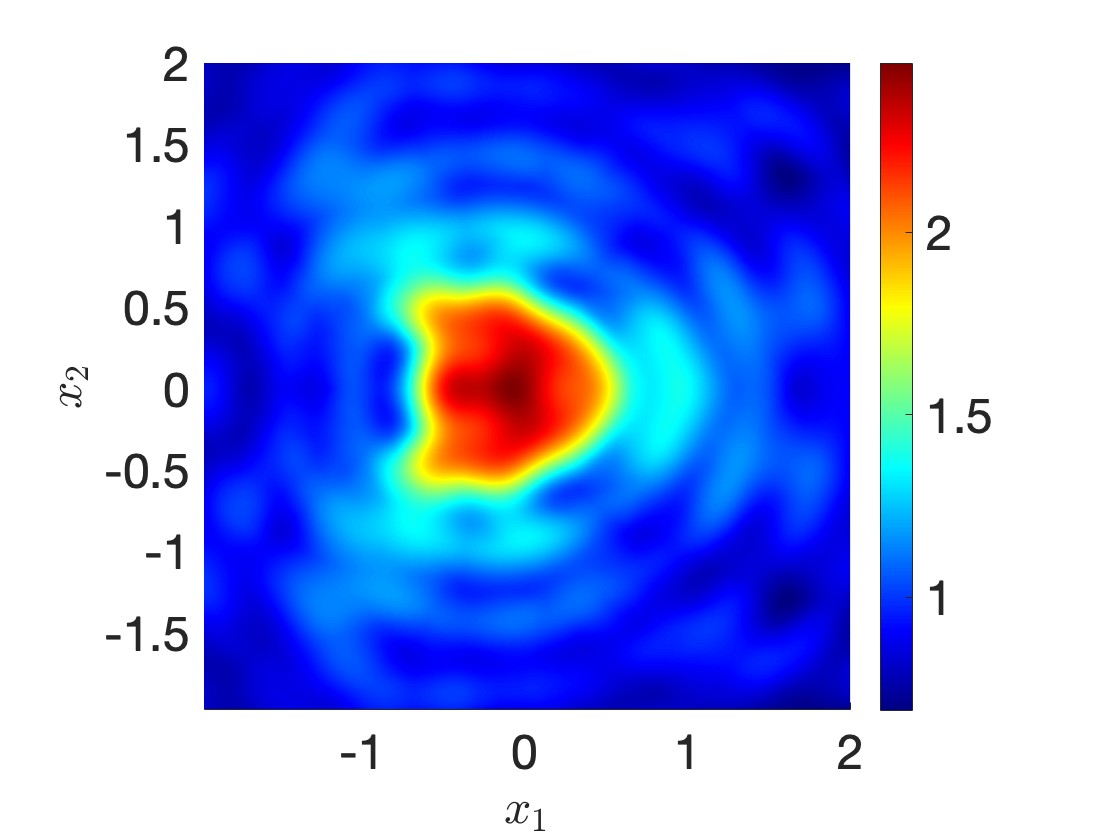}
        \caption{$\sum_{d} |\mathcal{I}|$ with $24$ evenly distributed $d \in \mathbb{S}^{1}$}
        \end{subfigure}
     \caption{True contrast $q$  and an estimate of its support  using $\mathcal{I}$.}
     \label{figI(z)}
\end{figure}
The next theorem demonstrates the stability of the imaging function against noise.
\begin{theorem}
 For $d\in \mathbb{S}^{n-1}$, we assume that $u^\delta(\cdot,d)$  satisfies   
\begin{align*}
\left\| |u(\cdot,d)| - |u^\delta(\cdot,d)|\right\|_{L^2(\partial \Omega)}  \leq \delta< 1.\end{align*}
Let $\mathcal{I}^{\delta}(z,d)$ denote the imaging function defined as in \eqref{Iz_phaseless}, with $|u(\cdot,d)|$ replaced by its noisy counterpart $|u^\delta(\cdot,d)|$. 
The following stability estimate holds
\begin{equation}
    \label{thm:stable}
    \left|\mathcal{I}(z,d)- \mathcal{I}^{\delta}(z,d) \right|\leq C \delta, \qquad \text{for all } z \in \Omega,
\end{equation}
where $C>0$ is independent of $z$ and $\delta$.
\end{theorem}

\begin{proof}
For $z\in \Omega$, we have
\begin{align*}
    |\mathcal{I}(z,d)- \mathcal{I}^{\delta}(z,d) | 
    & = \left| \int_{\partial \Omega}  \Psi(x,z) \left[ \frac{|u(x,d)|^{2}-|u^{in}(x,d)|^{2}}{\overline{u^{in}(x,d)}} - \frac{|u^{\delta}(x,d)|^{2}-|u^{in}(x,d)|^{2}}{\overline{u^{in}(x,d)}} \right] dS(x)  \right| \\
    & \leq \left|  \int_{\partial \Omega} \Psi(x,z)  \frac{|u(x,d)|^{2}-|u^{\delta}(x,d)|^{2}}{\overline{u^{in}(x,d)}} dS(x) \right|  \\
    & =  \left|  \int_{\partial \Omega} \frac{\Psi(x,z) }{\overline{u^{in}(x,d)}}  \left[ |u(x,d)|^{2}-|u^{\delta}(x,d)|^{2} \right] dS(x) \right| \\
     & \leq \|\Psi(\cdot,z)\|_{L^{2}(\partial \Omega)} \left\| \left(2|u(\cdot,d)|+(|u^{\delta}(\cdot,d)|- |u(\cdot,d)|)\right)\left(|u(\cdot,d)|-|u^{\delta}(\cdot,d)| \right)\right\|_{L^{2}(\partial \Omega)} \\
     &\le \|\Psi(\cdot,z)\|_{L^{2}(\partial \Omega)}\left[ 2\|u(\cdot,d)\|_{L^{2}(\partial \Omega)}\delta + \delta^2 \right] \\ 
     &\le \|\Psi(\cdot,z)\|_{L^{2}(\partial \Omega)}\left[ 2\|u(\cdot,d)\|_{L^{2}(\partial \Omega)}+1 \right]\delta.
\end{align*}
The last inequality is obtain using $\delta<1$. The proof is concluded by defining
\begin{align*}
    C := \sup_{z \in \Omega} \left( \|\Psi(\cdot,z)\|_{L^{2}(\partial \Omega)}\left[ 2\|u(\cdot,d)\|_{L^{2}(\partial \Omega)}+1 \right] \right).
\end{align*}

\end{proof}

\section{A system of approximate model equations}\label{section3}
In this section, we construct a system of approximate model equations. These equations are motivated by the results obtained in Section \ref{section2} and are combined with their representation in the Fourier domain to improve computational efficiency.

The system utilizes the equations \eqref{eqn:lippmannschwringerqn} and \eqref{approxI}. For \eqref{approxI}, when the measurement radius $R_\Omega$ is sufficiently large, the term $\mathcal{O}\left(R_\Omega^{(1-n)/2}\right)$ can be neglected. Consequently, we obtain that
$$\mathcal{I}(z,d) \approx k^{2} \int_{D} \mathrm{Im}\Phi(z,y)q(y)u(y,d) dy.$$
We will treat this approximation as an equation in the subsequent implementation. 
For $d\in \mathbb{S}^{n-1}$ and $ x,z \in \Omega$, the following equations relate the unknown contrast $q$ and total field $u$
\begin{align*}
     u(x,d) & = u^{in}(x,d) + k^{2}\int_{\Omega} \Phi(x,y)q(y)u(y,d)dy,\qquad x\neq y, \\
     \mathcal{I}(z,d) & = k^{2} \int_{D} \mathrm{Im}\Phi(z,y)q(y)u(y,d) dy.
\end{align*}
Multiplying the first equation by $q$ and denoting $qu(\cdot,d)$ as $\omega(\cdot,d)$, we obtain the system of equations
\begin{align}\label{eq:stateeqn}
     \omega(x,d)  & = q(x) u^{in}(x,d) + k^{2}q(x)\int_{\Omega} \Phi(x,y)\omega(y,d) dy,\qquad x\neq y,  \\ \label{eq:dataeqn}
     \mathcal{I}(z,d) & = k^{2} \int_{D} \mathrm{Im}\Phi(z,y)\omega(z,d)  dy.
\end{align}
The equations \eqref{eq:stateeqn} and \eqref{eq:dataeqn} will be referred to as the state equation and the data equation, respectively.

\begin{remark}
  The system \eqref{eq:stateeqn}-\eqref{eq:dataeqn} can be adopted as the model system for the inverse scattering problem with phased boundary measurement data. In this setting, the corresponding imaging function is given by $\widetilde{\mathcal{I}}_d$, so \eqref{eq:dataeqn} holds as an exact equation. Consequently, in the presence of phased data, the system constitutes an exact formulation rather than an approximation, particularly with respect to the data equation.
\end{remark}
The system of state - data equations \eqref{eq:stateeqn}-\eqref{eq:dataeqn} can be considered as a set of model equations. However, direct discretization of this equation generally results in a large number of parameters, especially in three dimensions.  To improve computational efficiency, the equation is transformed into the spectral domain, with functions expressed through their Fourier coefficients. In this setting, the dominant behavior of the solution is often captured by a small subset of Fourier modes, allowing for a reduced computational cost without compromising essential characteristics. Moreover, turning the state equation allows us to avoid the singularity of the Green's function $\Phi$ in the state equation.
To this end, we employ a periodization technique for equations \eqref{eq:stateeqn} and \eqref{eq:dataeqn}. An additional computational advantage of this approach is that it transforms the equations into algebraic relations in the Fourier domain, where the Fourier coefficients of the periodized kernels can be computed explicitly. This periodization technique is motivated by the work in \cite{vainikko2000fast}.


Let ${B_0(\rho)}$ be the ball with radius $\rho$ that is centered at the origin such that $\text{supp}(q)\subset \overline{{B_0(\rho)}}$. We define
$$\mathcal{P} (x) = \begin{cases}
    \Phi(x),\hspace{0.7 cm} |x|\le R, \\ \notag
    0,\hspace{1.4 cm} x\in G_R\setminus \overline{B_0(R)}, 
\end{cases}\qquad R\ge 2\rho,$$
where $G_R$ is the open cube centered at the origin with side length $2R$,
$$G_R=\{x=(x_1,\ldots,x_n)\in \R^n: |x_\ell|<R, \, \ell=1,\ldots,n \}.$$
Then we extend  $\mathcal{P}$ from $G_R$ to $\R^n$ as a $2R$-periodic function. The contrast $q$ and contrast source function $\omega$ are also periodized in the same way, obtaining $q_{\text{\text{per}}}$  and $\omega_{\text{\text{per}}}$, respectively. By this, the periodized version of the system of equations \eqref{eq:stateeqn}-\eqref{eq:dataeqn} is given by
\begin{align*}
    \omega_{\text{\text{per}}}(x,d)  & = q_{\text{\text{per}}}(x) u^{in}(x,d) + k^{2}q_{\text{\text{per}}}(x)\int_{G_{R}} \mathcal{P}(x-y)\omega_{\text{\text{per}}}(x,d) dy, \qquad x\neq y, \\ 
    \mathcal{I}(z,d) & = k^{2} \int_{G_{R}} \mathcal{K}(z-y)\omega_{\text{\text{per}}}(y,d)  dy,
\end{align*}
where $x,z \in \Omega$ and $\mathcal{K} = \mathrm{Im}(\mathcal{P})$. 
We next reformulate this system in Fourier domain. First,
recall that the trigonometric orthonormal basis of $L^2(G_R)$ is given by
$$\varphi_j(x)=(2R)^{-n/2} e^{i\pi j\cdot\frac{x}{R}},\qquad j=(j_1,\ldots,j_n)\in \Z^n.$$
For $g\in L^2(G_R)$ and $j\in \Z^n$, the Fourier coefficients of $g$ are given by
$$\widehat{g}(j)=\int_{G_R} g(x)\overline{\varphi_j(x)}dx.$$
Now, observe that  the right-hand sides of the equations exhibit a convolution structure. From this, we derive the \textbf{set of model equations}
\begin{align} \label{eq:perstateeqn_model}
\mathcal{F}^{-1}[\widehat{\omega}_{\text{per}}(j,d)](x,d) 
&= q_{\text{per}}(x)\, u^{in}(x,d)
   + k^{2} q_{\text{per}}(x)(2R)^{n/2}\mathcal{F}^{-1}\!\left[\, \widehat{\mathcal{P}}(j)\, \widehat{\omega}_{\text{per}}(j,d)\right](x),
\qquad x \in \Omega. 
\\ \label{eq:perdataeqn_model}
\mathcal{I}(z,d)
&= k^{2}\, \,(2R)^{n/2}\mathcal{F}^{-1}\!\left[\, \widehat{\mathcal{K}}(j)\, \widehat{\omega}_{\text{per}}(j,d)\right](z),
\hspace{4.9cm} z \in B_{0}(\rho). 
\end{align}
where $\mathcal{F}$ is the Fourier transform.
To use this system of model equations, we need explicit Fourier coefficients of $\mathcal{K}$ and $\mathcal{P}$ and their behavior, which we now present. 
 

\begin{theorem}\label{thm:Fcoef_K} For $j\in \Z^n$,
    the Fourier coefficients of $\K$ are given as follows.
    \begin{itemize}
        \item For $n=2$,
        \begin{equation}\notag
            \widehat{\K}(j) = \begin{cases}
        \dfrac{\pi R}{8} \left[J_0^2(kR)+J_1^2(kR)\right], &  \pi |j|= kR, \\ 
            -\dfrac{1}{4}\cdot\dfrac{\pi R}{k^2R^2-\pi^2|j|^2}\left[\pi|j|J_0(kR)J_1(\pi|j|)-kR J_1(kR)J_0(\pi|j|)\right],& \pi |j|\neq kR.
        \end{cases} 
        \end{equation}
       
        \item For $n=3$,
        \begin{equation}\notag
            \widehat{\K}(j) = \begin{cases}
        (2R)^{-3/2}\dfrac{ R^2}{2}\left\{ kR \left[j_0^2(kR)+j_1^2(kR)\right]-j_0(kR)j_1(kR)\right\}, &  \pi |j|= kR, \\
           kR(2R)^{-3/2}\dfrac{R^2}{k^2R^2-\pi^2|j|^2}\left[kRj_0(\pi |j|)j_1(kR)-\pi|j|j_0(kR)j_1(\pi|j|)\right],&\pi |j|\neq kR,|j|\neq 0,\\ 
           \dfrac{1}{k}(2R)^{-3/2}\left(\dfrac{1}{k}\sin(kR)-R\cos(kR) \right),& |j|=0.
        \end{cases} 
        \end{equation}
    \end{itemize}
     Consequently, as $|j|\to \infty$, 
    \begin{equation}\notag
        \widehat{\K}(j) =  \begin{cases}
    \mathcal{O}\left(|j|^{-3/2}\right), & \text{if }  n=2,\\ 
    \mathcal{O}\left(|j|^{-2}\right), &\text{if }  n=3.
\end{cases}
    \end{equation}
Furthermore, for all $j\in \Z^n$ such that $\pi|j|\ge kR \ge 1$, we have $|\widehat{\K}(j)|\le \mu,$ where
$$\mu = \begin{cases}
    \frac{\pi R}{8} \left[J_0^2(kR)+J_1^2(kR)\right], & \text{if } n=2, \\ \notag
    (2R)^{-3/2}\frac{ R^2}{2}\left[ kR [j_0^2(kR)+j_1^2(kR)]-j_0(kR)j_1(kR)\right],& \text{if } n=3.
\end{cases}$$
\end{theorem}
\begin{proof}
    The proof of this theorem can be found in \cite{submitted} and is therefore omitted here.
\end{proof}
Besides providing an explicit representation of $\widehat{\K}(j)$, the theorem indicates that $|\widehat{\K}(j)|$ attains its maximum value on the ring $|j|=kR/\pi$. Moreover, $|\widehat{\K}(j)|$ decays sharply as $|j|$ moves away from that ring. As shown in Figure \ref{fig:kernel-K-hat}, the dominant coefficients of $|\widehat{\K}(j)|$ are concentrated along the ring $|j|=kR/\pi$, where they attain significantly larger magnitudes, while rapidly decreasing away from it.
\begin{figure}[H]
    \centering
    \begin{subfigure}[b]{0.3\textwidth}
            \includegraphics[width=1.0\linewidth]{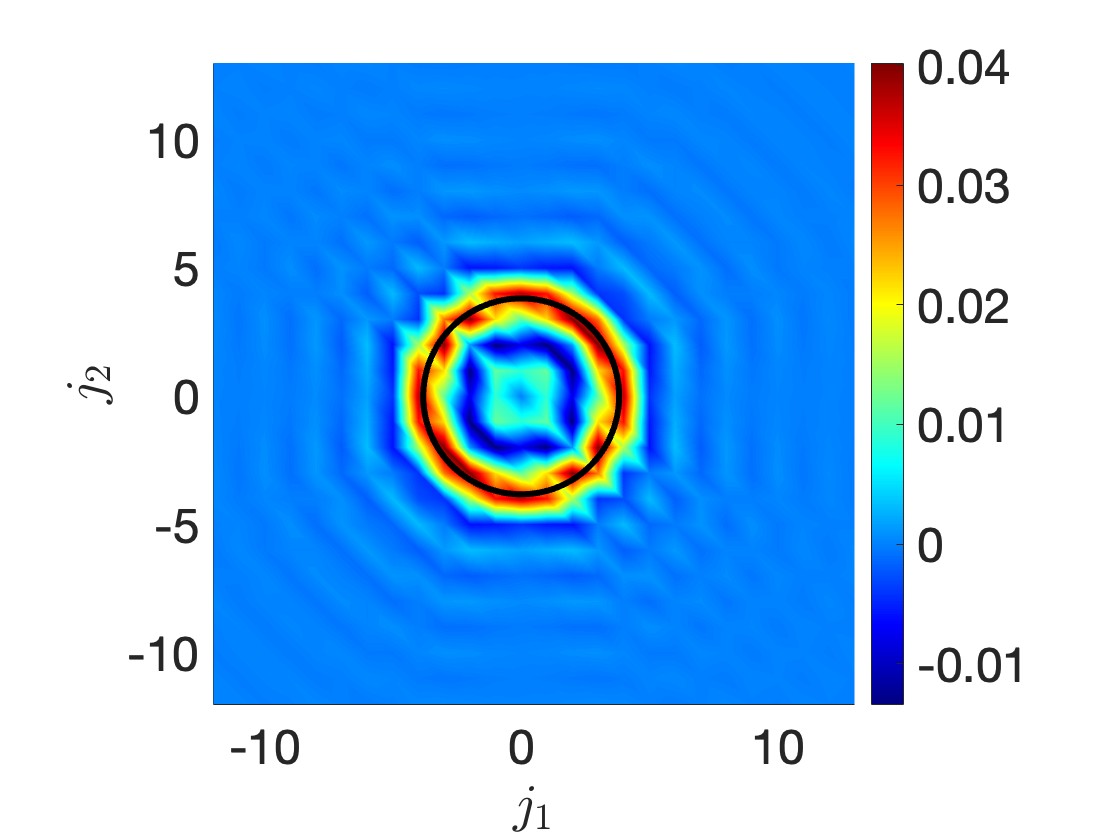}
            \caption{$k=6$}
        \end{subfigure}
    \begin{subfigure}[b]{0.3\textwidth}
            \includegraphics[width=1.0\linewidth]{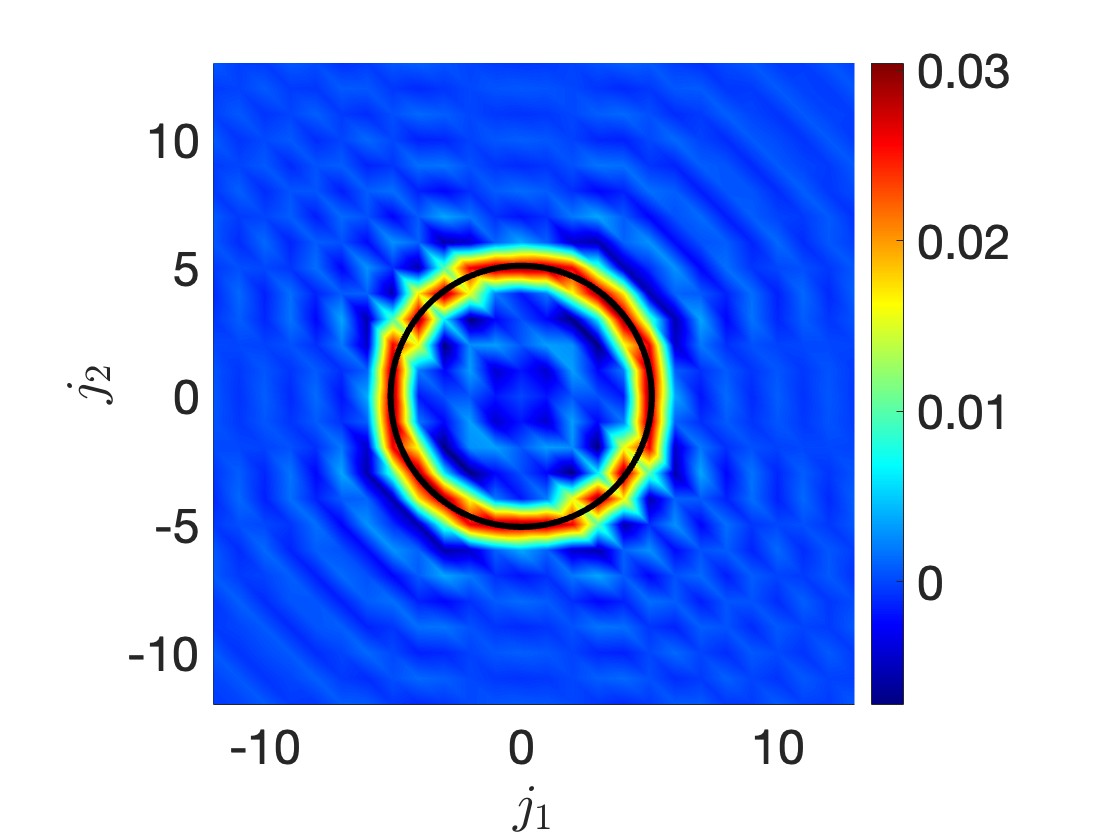}
            \caption{$k=8$}
        \end{subfigure}
   \begin{subfigure}[b]{0.3\textwidth}
            \includegraphics[width=1.0\linewidth]{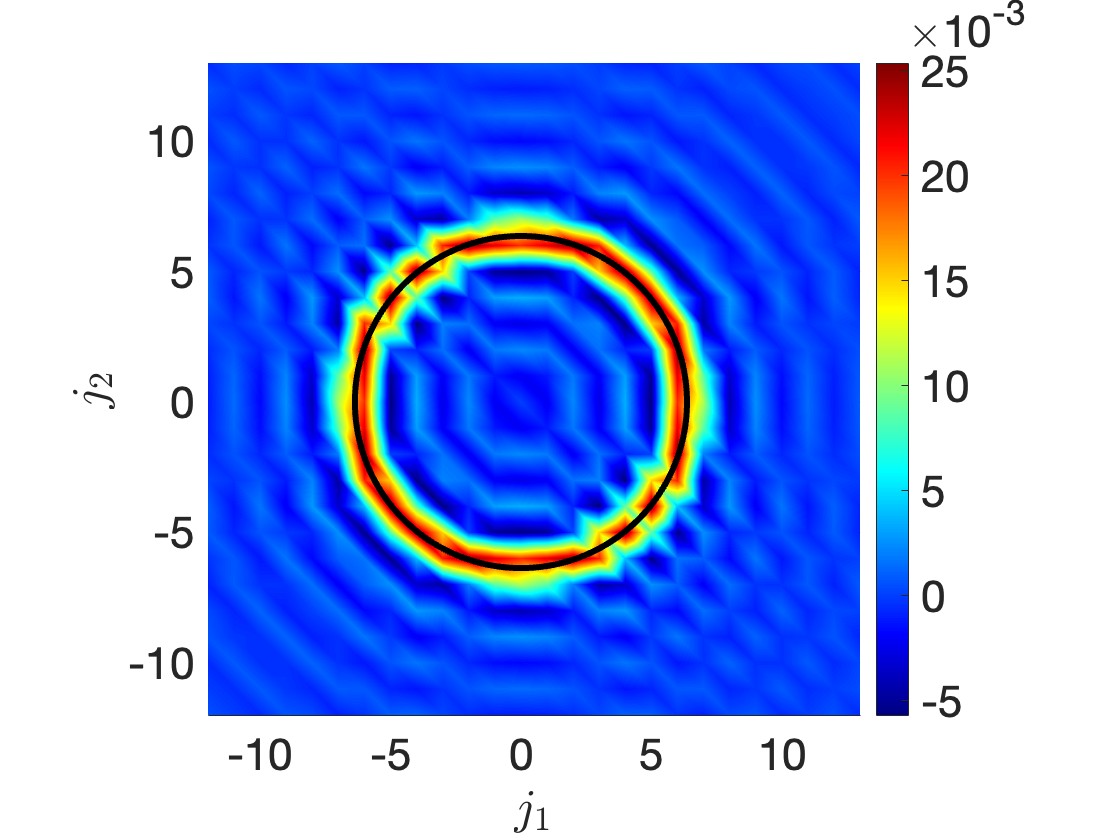}
            \caption{$k=10$}
        \end{subfigure}
    \caption{$\left| \widehat{\mathcal{K}} \right|$ with the circle $|j|=kR/\pi$ (black ring) for $n=2$
    }\label{fig:kernel-K-hat}

\end{figure}

The following theorem presents a detailed analysis of Fourier coefficients $\widehat{\mathcal{P}}$. 

\begin{theorem}\label{thm:Fcoef}
     For $j\in \Z^n$, the Fourier coefficients of $\mathcal{P}$ are given as follows.
    \begin{itemize}
        \item For $n=2$,
        \begin{equation}\label{Fco:n2}
 \widehat{\mathcal{P}}(j) = \begin{cases}
        \dfrac{i\pi R}{8}\left[ H_0^{(1)}(kR) J_0(kR)+ H_1^{(1)}(kR)J_1(kR) \right], \hspace{4.6 cm}  \pi |j|= kR, \\ 
           \dfrac{R}{2(k^2R^2-\pi^2|j|^2)}\left\{-1+\dfrac{i\pi}{2}\left[-\pi |j| H_0^{(1)}(kR) J_1(\pi|j|)+kR H_1^{(1)}(kR)J_0(\pi|j|) \right]\right\},\\ 
           \hspace{12 cm}\pi |j|\neq kR, j\neq 0, \\ 
           \dfrac{-1}{2k^2 R} + \dfrac{i\pi}{4k}H_1^{(1)}(kR), \hspace{8.4 cm} j=0.
        \end{cases} 
        \end{equation}

        \item For $n=3$,
        \begin{equation}\label{Fco:n3}
            \widehat{\mathcal{P}}(j) = \begin{cases}
        \dfrac{i R}{2k}(2R)^{-3/2}\left[ 1-\dfrac{e^{ikR}}{kR}\sin(kR) \right], & \pi |j|= kR, \\
           \dfrac{R^2}{k^2R^2-\pi^2|j|^2}(2R)^{-3/2}\left\{-1+e^{ikR}\left[ \cos(\pi|j|)-\dfrac{ikR}{\pi|j|}\sin(\pi|j|) \right]\right\},& \pi |j|\neq kR,|j|\neq 0,\\ 
            \dfrac{1}{k^2}(2R)^{-3/2}\left( -1+e^{ikR}(1-ikR)\right),& j=0.
        \end{cases} 
        \end{equation}
    \end{itemize}
    Moreover, if $\pi|j|>kR\ge 1,$ then we have 
    $$|\widehat{\mathcal{P}}(j)| \le \begin{cases}
       \dfrac{C(\pi |j|)^{1/2}}{\pi^2|j|^2 - k^2R^2} ,&\text{if } n=2, \\ \notag
        \dfrac{R^2}{\pi^2|j|^2 -k^2 R^2}\cdot\dfrac{3}{(2R)^{3/2}},&\text{if } n=3.
    \end{cases}$$
    where $C>0$.
    Consequently, as $|j|\to \infty$,
    \begin{equation}\label{assp:P}
        |\widehat{\mathcal{P}}(j)| = \begin{cases}
    \mathcal{O}\left(|j|^{-3/2}\right), & \text{if }  n=2,\\ 
    \mathcal{O}\left(|j|^{-2}\right),&\text{if }  n=3.
    \end{cases}
    \end{equation}
\end{theorem}

\begin{proof}
 The derivation of Fourier coefficients $\widehat{\mathcal{P}}$ can be found in \cite{vainikko2000fast}. 
Regarding the boundedness of $|\widehat{\mathcal{P}}(j)| $, assume that
 $\pi|j|>kR\ge 1$. When $n=2$, since 
 $|H_\alpha^{(1)}(\pi|j|)|$ is  bounded
 when $\pi|j|\ge 1$ and  $|J_\alpha(\pi|j|)|=\mathcal{O}(\pi|j|^{-1/2})$ as $|j|\to \infty$ for $\alpha=0,1$,
 we have
      \begin{align}
        \notag
        |\widehat{\mathcal{P}}(j)|& =  \dfrac{R}{2(\pi^2|j|^2 - k^2R^2)} \left|-1 + \dfrac{i\pi}{2}\left[-\pi|j|H_0^{(1)}(kR)J_1(\pi|j|) + kRH_1^{(1)}(kR)J_0(\pi|j|) \right] \right| \\ \notag
        & \le \dfrac{R}{2(\pi^2|j|^2 - k^2R^2)}  \left\{ 1+\dfrac{\pi}{2}\left[\pi|j||H_0^{(1)}(kR)||J_1(\pi|j|)| + kR|H_1^{(1)}(kR)||J_0(\pi|j|)| \right]\right\} \\ \notag
        & \le \dfrac{R}{2(\pi^2|j|^2 - k^2R^2)} \left[1+C'((\pi|j|)^{1/2}+(\pi|j|)^{-1/2})\right]\\ \notag
        &\le \dfrac{C(\pi|j|)^{1/2}}{\pi^2|j|^2 - k^2R^2}. 
    \end{align}
    where $C',C>0$.
    Similarly, when $n=3$,
    \begin{align}
        \notag
        |\widehat{\mathcal{P}}(j)|& = \dfrac{R^2}{\pi^2|j|^2 - k^2R^2} \dfrac{1}{(2R)^{3/2}}\left| -1 + e^{ikR}\left[\cos(\pi |j|) - \dfrac{ikR}{\pi|j|}\sin(\pi|j|)\right]\right| \\ \notag
        & \le \dfrac{R^2}{\pi^2|j|^2 - k^2R^2} \dfrac{1}{(2R)^{3/2}}\left\{ 1 + \left[|\cos(\pi |j|)| + \dfrac{kR}{\pi|j|}|\sin(\pi|j|)|\right]\right\} \\ \notag  
        & \le \dfrac{R^2}{\pi^2|j|^2 - k^2R^2}\cdot \dfrac{3}{(2R)^{3/2}}.
    \end{align}
The asymptotic behavior of $|\widehat{\mathcal{P}}(j)|$ in \eqref{assp:P} can be directly deduced from the estimates above.
\end{proof}

The theorem above demonstrates that $|\widehat{\mathcal{P}}(j)|$ decays rapidly, with a rate of with the rate of $|j|^{-3/2}$ for $n=2$ and of $|j|^{-2}$ for $n=3$ as $|j|\to\infty$. Although the maximum value of $|\widehat{\mathcal{P}}(j)|$ is not established in this work, we are able to indicate that $|\widehat{\mathcal{P}}(j)|$ can still be controlled by a strictly decreasing upper bound, despite its oscillatory behavior arising from the properties of  Bessel functions. Furthermore, these bounds attain relatively large values in a neighborhood of the ring $|j| = kR/\pi$. Hence, it is reasonable to expect that $|\widehat{\mathcal{P}}(j)|$ attains large values when $|j| \approx kR/\pi$. This behavior is illustrated in Figure \ref{fig:kernel-hat}.

\begin{figure}[H]
    \centering
    \begin{subfigure}[b]{0.3\textwidth}
            \includegraphics[width=1.0\linewidth]{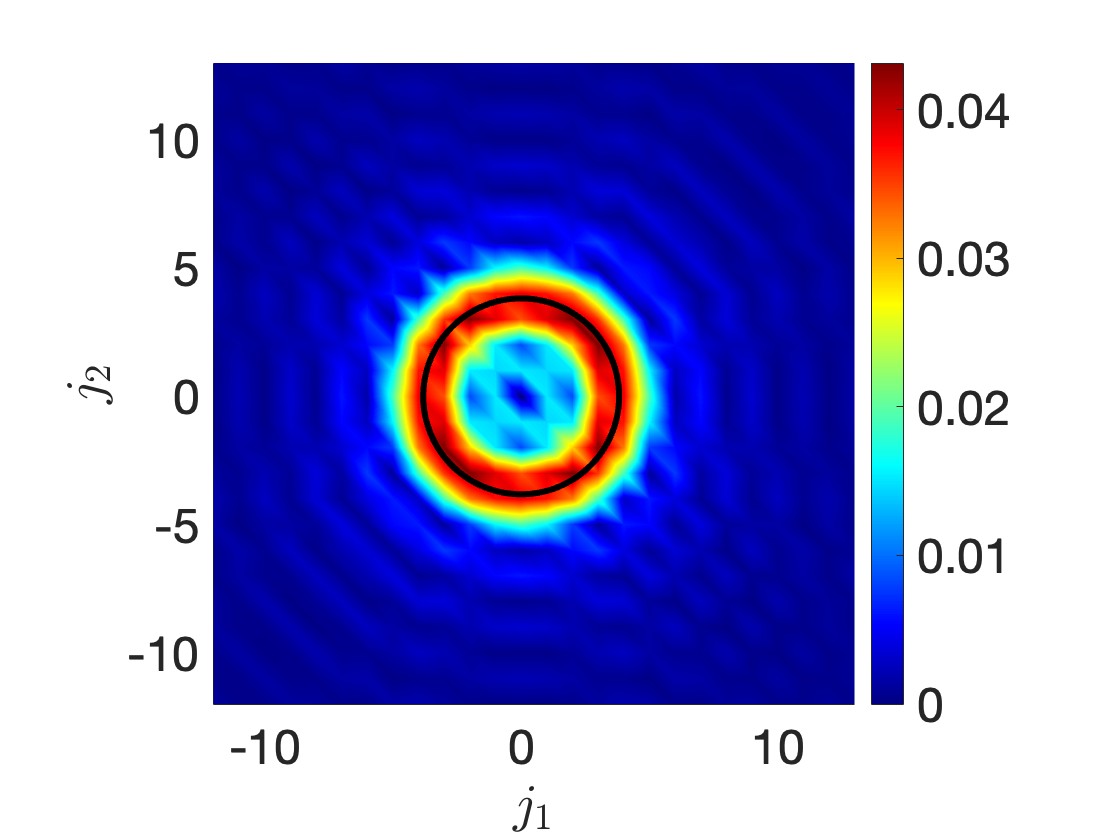}
            \caption{$k=6$}
        \end{subfigure}
    \begin{subfigure}[b]{0.3\textwidth}
            \includegraphics[width=1.0\linewidth]{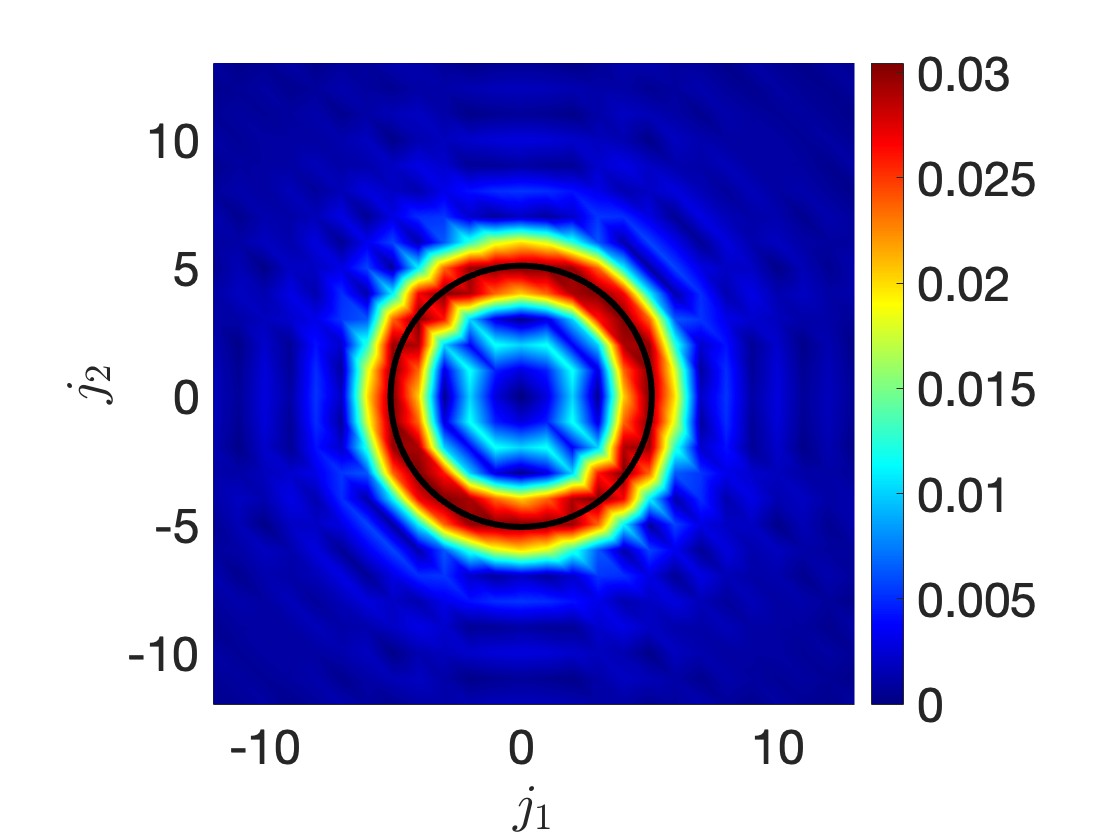}
            \caption{$k=8$}
        \end{subfigure}
   \begin{subfigure}[b]{0.3\textwidth}
            \includegraphics[width=1.0\linewidth]{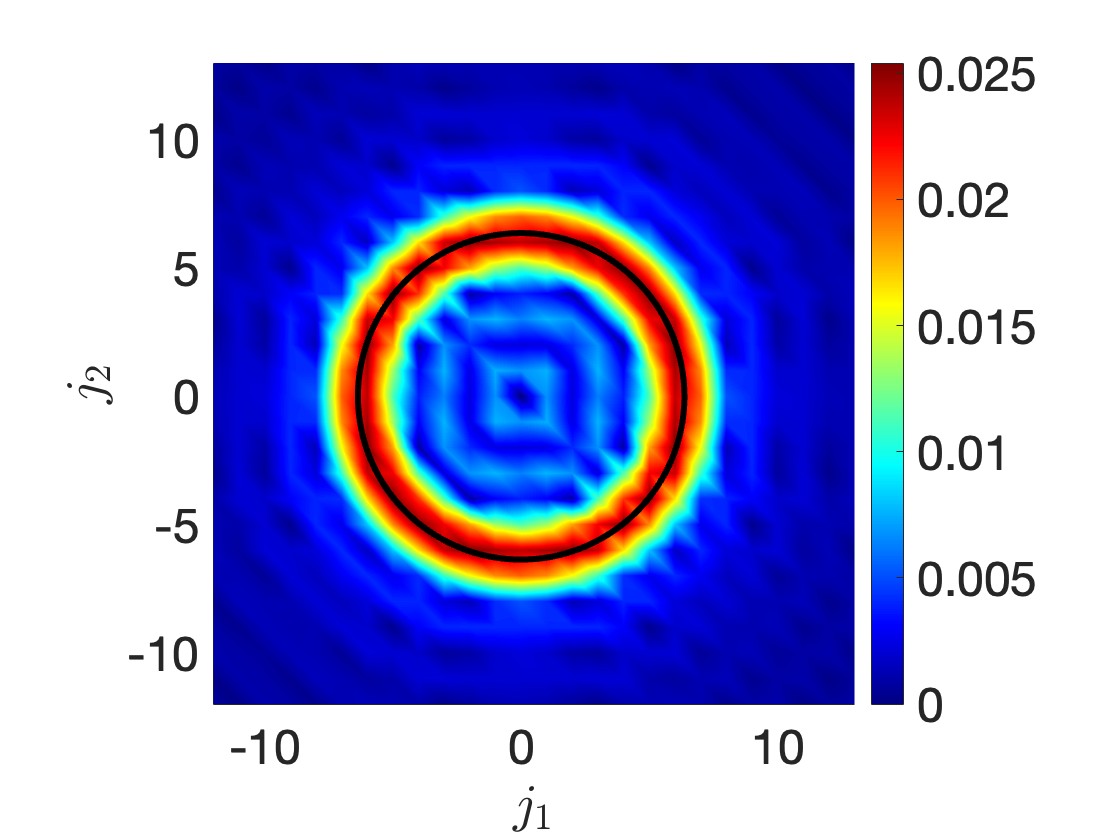}
            \caption{$k=10$}
        \end{subfigure}
    \caption{$\left| \widehat{\mathcal{P}}\right|$ with the circle $|j|=kR/\pi$ (black ring) for $n=2$
    }\label{fig:kernel-hat}

\end{figure}

The characterization of the dominant coefficients of both $\widehat{\mathcal{P}}(j)$ and $\widehat{\mathcal{K}}(j)$ informs the selection of an appropriate truncation level in the Fourier series representation when implementing the system \eqref{eq:perstateeqn_model}-\eqref{eq:perdataeqn_model}. This provides a lower bound reduction in the number of Fourier coefficients required by the model to capture the dominant coefficients of $\omega$, and hence plays a crucial role in improving the computational efficiency of the neural network training procedure, as discussed in the next section. In particular, we truncate the series to the index set 
$j \in [-J+1,J]^n \subset \Z^n$, 
where $J$ is taken larger but not too far from $kR/\pi$.

\section{Model-based deep learning algorithm}\label{section4}


In this section, we present  a formulation of  the proposed model-informed deep learning algorithm. The central principle of this approach is the optimization of parametric hypothesis spaces governed by the system of equations \eqref{eq:perstateeqn_model}-\eqref{eq:perdataeqn_model}. To  formulate the optimization problem, we specify the hypothesis spaces $\mathcal{H}
$ and $\widehat{\mathcal{H}}$, along with the corresponding objective function.


This optimization problem requires two hypothesis spaces: the first hypothesis space $\widehat{\mathcal{H}}$ approximates the Fourier coefficients of the contrast source function $\widehat{\omega}_{\Phi}$, and the second hypothesis space $\mathcal{H}$ approximates the contrast function $q_{\Theta}$ in free space. Both hypothesis spaces are taken to be the class of fully-connected feedforward neural networks, wherein the parameter sets $\Phi$ and $\Theta$ are independent.


The optimization procedure is predicated on the following assumptions:
\begin{enumerate}
    \item There exist models $q_{\Theta} \in \mathcal{H}$ and $\widehat{\omega}_{\Phi} \in \widehat{\mathcal{H}}$ for which the loss function $\mathcal{L}$ defined in \eqref{eq:istloss} attains a sufficiently small value.
    \item Under the preceding assumption, the corresponding relative error $\mathcal{E}$ defined in \eqref{eq:total error} is likewise small.
    \end{enumerate}


The ultimate objective is to identify a model $q_{\Theta}^{\ast} \in \mathcal{H}$ such that $q_{\Theta}^{\ast} \approx q_{\text{\text{per}}}$. In conjunction with the hypothesis spaces $\mathcal{H}$ and $\widehat{\mathcal{H}}$, the system of model equations \eqref{eq:perstateeqn_model}-\eqref{eq:perdataeqn_model} gives rise to the model-based loss function
\begin{align}\label{eq:istloss}
    \mathcal{L}: \widehat{\mathcal{H}} \times \mathcal{H} & \rightarrow \mathbb{R}  \\
(\widehat{\omega}_{\Phi},q_{\Theta}) & \mapsto \frac{1}{M}\sum_{m=1}^{M} \bigg( \left| \mathcal{I}(z,d_m) - k^{2}(2R)^{n/2}\mathcal{F}^{-1} \left[\widehat{\K}(j)\widehat{\omega}_{\Phi}(j,d_{m}) \right](z,d_{m})\right|^{2} \nonumber \\
& \hspace{-0.5cm}+  \left| \mathcal{F}^{-1} \left[\widehat{\omega}_{\Phi}(j,d_{m}) \right](x,d_m)  - q_{\Theta}(x)u^{in}(x,d_{m}) -k^{2}q_{\Theta}(x) (2R)^{n/2}  \mathcal{F}^{-1} \left[\widehat{\mathcal{P}}(j) \widehat{\omega}_{\Phi}(j,d_{m}) \right](x,d_{m}) \right|^{2} \bigg), \nonumber
\end{align}
where $j\in \mathbb{Z}^n, \, x \in \Omega,$ and $  z\in B_{0}(\rho)$.



With the loss function $\mathcal{L}$, the training objective is formulated as the following minimization problem
\begin{align*}
    \min_{\widehat{\omega}_{\Phi} \in \widehat{\mathcal{H}}, \, q_{\Theta} \in \mathcal{H}} \mathcal{L}[\widehat{\omega}_{\Phi},q_{\Theta} ] & := \min_{\widehat{\omega}_{\Phi} \in \widehat{\mathcal{H}}, \, q_{\Theta} \in \mathcal{H}}  \bigg\{
    \frac{1}{M}\sum_{m=1}^{M} \bigg( \left| \mathcal{I}(z,d_m) - k^{2}(2R)^{n/2}\mathcal{F}^{-1} \left[\widehat{\K}(j)\widehat{\omega}_{\Phi}(j,d_{m}) \right](z,d_{m})\right|^{2} \nonumber \\
 &\hspace{-2.5cm}+  \left| \mathcal{F}^{-1}  \left[\widehat{\omega}_{\Phi}(j,d_{m}) \right](x,d_m)   - q_{\Theta}(x)u^{in}(x,d_{m}) -k^{2}q_{\Theta}(x) (2R)^{n/2}  \mathcal{F}^{-1} \left[\widehat{\mathcal{P}}(j) \widehat{\omega}_{\Phi}(j,d_{m}) \right](x,d_{m}) \right|^{2} \bigg)
\bigg\},
\end{align*}
where $j\in \mathbb{Z}^n, \, x \in \Omega,$ and $  z\in B_{0}(\rho)$.

Two neural networks are initialized in parallel: one to approximate $\widehat{\omega}_{\Phi}$
, which operates in the frequency domain and takes the frequency variables $j_{1}, j_{2}, j_{3}$ and the unit direction vector $ d^{m}_{1}, d^{m}_{2}, d^{m}_{3}$ as inputs, and one to approximate $q_{\Theta}$, which operates in the physical domain and takes the spatial variables $x_{1}, x_{2}, x_{3}$ as inputs. Both networks are optimized concurrently at each iteration via simultaneous gradient updates $\Theta \leftarrow \Theta - \tau\nabla_{\Theta}\mathcal{L}$, $\Phi \leftarrow \Phi- \tau\nabla_{\Phi}\mathcal{L}$, sharing a common learning rate $\tau$. The total loss $\mathcal{L}$ is defined as the equal-weighted sum of the data equation loss and the state equation loss. To minimize the data equation loss, $\widehat{\omega}_{\Phi}$ is learned using the imaging function $\mathcal{I}(z, d_m)$, which is produced in the first stage and serves as a direct input to the data equation loss term, encoding prior information about the scatterer's shape and location. This estimate then informs the recovery of $q_{\Theta}$ through the state equation loss. That is, $\widehat{\omega}_{\Phi}$, obtained by minimizing the data equation loss, is substituted into the state equation loss to recover $q_{\Theta}$. The objective is to find parameters $\Phi$ and $\Theta$ that jointly minimize $\mathcal{L}$, as described in Algorithm~\ref{alg:Model_Informed_Neural_Network_Algorithm}.

A schematic for  the training process for $n=3$ can be seen in Figure~\ref{fig:Training Process}, where $N$ and $\tilde{N}$ denotes the number of points in the discretization of $\Omega$ and $B_{0}(\rho)$ respectively. 

\paragraph{Choice of hyperparameters.} The neural network approximating $\widehat{\omega}_{\Phi}$ takes as input the set of quadrature points
$$\{j=(j_1,j_2)^T\in \mathbb{Z}^{2}: j_1,j_2 = -12,-11, ..., 12,13\}, \qquad\{d_m=(d_1^m,d_2^m)^T\in\mathbb{S}^{1}: m = 1,2,...,M\},$$
for $n=2$ and $$\{j=(j_1,j_2,j_3)^T\in \mathbb{Z}^{3}: j_1,j_2,j_3 = -11, ..., 12\}, \qquad\{d_m=(d_1^m,d_2^m,d_3^m)^T\in\mathbb{S}^{2}: m = 1,2,...,M\},$$
for $n=3$, with complex-valued output. The neural network approximating $q_{\Theta}$ takes as input the set of sampling points obtained by uniformly discretizing $[-2,2]^{n}$ into $64$ points along each axis, with output dimension corresponding to $1$. Both networks share the same architecture: depth $4$, width $64$, and the $\tanh$ activation function in the hidden layers. The parameters of each network are initialized independently using the Glorot normal initializer, and each network is trained via a dedicated Adam optimizer with a learning rate of $0.001$, for $100,000 $ iterations in the case $n=2$ and $70,000$ iterations in the case $n=3$. Convergence of the algorithm is assessed by monitoring the stabilization of the relative error $\mathcal{E}$ in \eqref{eq:total error}. In the numerical experiments presented below, the relative error is observed to stabilize after $70,000$ iterations (see Figure \ref{fig:Update process and Loss curve (2D) - Kite}). 

\begin{figure}[H]
    \centering
    \includegraphics[scale=0.35]{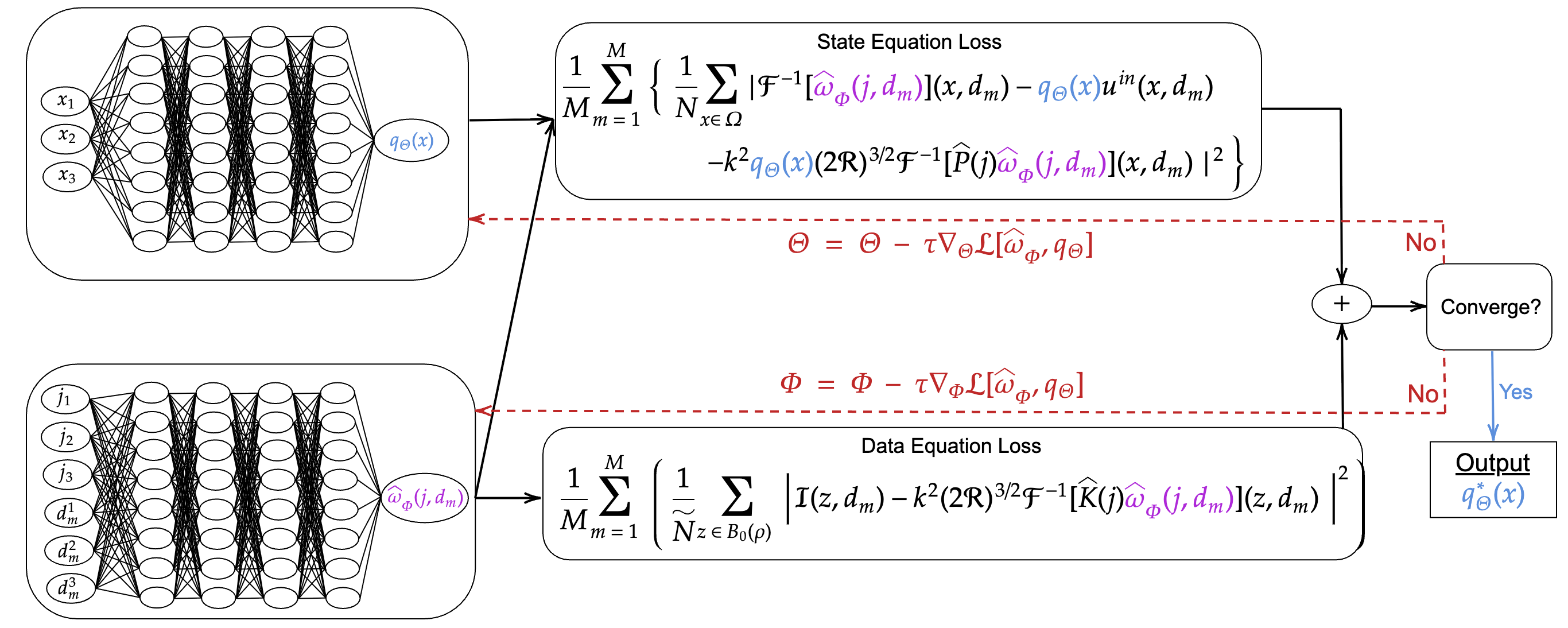}
    \caption{A schematic for the training process as $n = 3$.}
    \label{fig:Training Process}
\end{figure}



\begin{algorithm}[H]
\caption{Model-informed Deep Learning Algorithm}
\label{alg:Model_Informed_Neural_Network_Algorithm}
\SetKwInOut{Input}{Input}
\Input{$N$: number of collocation points in $\Omega$ \\
$\tilde{N}$: number of collocation points in $B_0(\rho)$ \\
\hspace*{-0.2em}$M$: number of incident directions\\
\hspace*{-0.7em}$iter$: number of iterations \\
\hspace*{0.3em}$\tau$: learning rate}
\BlankLine
Initialize the hypothesis spaces $\widehat{\mathcal{H}}$ and $\mathcal{H}$ with the hyperparameters given above.\\
\For{$i=1:\text{iter}$}{
Compute:
\vspace{-0.5cm}
\begin{align*}
\mathcal{L}(\widehat{\omega}_{\Phi}, q_{\Theta})
& = \frac{1}{M}\sum_{m=1}^{M} \bigg( \frac{1}{N}\sum_{z \in B_{0}(\rho)}
\left| \mathcal{I}(z,d_m) - (2R)^{n/2}\mathcal{F}^{-1}\!\left[\widehat{\mathcal{K}}(j)\,\widehat{\omega}_{\Phi}(j,d_{m})\right]\!(z,d_{m})\right|^{2} \\
&\hspace{-1.5 cm}+\frac{1}{\tilde{N}}\sum_{x \in \Omega}\,\left| \mathcal{F}^{-1}\!\left[\widehat{\omega}_{\Phi}(j,d_{m})\right]\!(x,d_{m}) - q_{\Theta}(x)\,u^{\mathrm{in}}(x,d_{m}) - k^{2}q_{\Theta}(x)\,(2R)^{n/2}\mathcal{F}^{-1}\!\left[\widehat{\mathcal{P}}(j)\,\widehat{\omega}_{\Phi}(j,d_{m})\right]\!(x,d_{m})\right|^{2}
\bigg);
\end{align*}

Update $\Phi \leftarrow \Phi - \tau\nabla_{\Phi}\,\mathcal{L}(\widehat{\omega}_{\Phi}, q_{\Theta})$\; 

Update $\Theta \leftarrow \Theta - \tau\nabla_{\Theta}\,\mathcal{L}(\widehat{\omega}_{\Phi}, q_{\Theta})$\;
}
\SetKwInOut{Output}{Output}
\Output{$ q^{\ast}_{\Theta}$}
\end{algorithm}

\section{Numerical results}\label{section5}


In this section, we illustrate the effectiveness of our proposed Algorithm \ref{alg:Model_Informed_Neural_Network_Algorithm} in reconstructing the unknown scatterer from boundary field data with numerical examples across the $2$D and $3$D settings. For the $2$D case, the wave number is set to be $7$ and the number of incident waves is $24$. The measurement boundary $\partial \Omega$ is taken to be the circle of radius $50$, centered at $0$, uniformly discretized into $64$ points. For the $3$D case, the wave number is set to be $7$ and the number of incident waves is $16$. The measurement boundary $\partial \Omega$ is taken to be the sphere of radius $50$, centered at $0$, uniformly discretized into $24$ azimuthal and $24$ polar points. The scattered field data on the boundary domain, $u^{sc}$, is generated via the Lippmann-Schwinger equation \eqref{eqn:lippmannschwringerqn}, from which the phaseless data, $u_{p},$ is subsequently computed using \eqref{phaselessformula}. To demonstrate the stability of our algorithm, $10\%$ noise is added to the boundary data according to the formula
\[u_p = u_p + 0.1 \times \frac{\mathcal{N}}{\left\Vert\mathcal{N}\right\Vert_F} \left\Vert{u_p}\right\Vert_F,\]
where $\mathcal{N}$ is the noise matrix whose entries are of the form $a + bi$ with $a, b \in (-1, 1)$ drawn independently from an uniform distribution, and $|| \cdot ||_{F}$ denotes the Frobenius norm.


The sampling domain $\Omega$ is taken to be the grid $[-2,2]^{n}$ uniformly discretized into $64$ points along each axis. To quantitatively assess the performance of the algorithm, we introduce the relative error metric
\begin{align}\label{eq:total error}
    \mathcal{E} \left[ q_{\Theta}^{\ast}\right] & = \frac{  \left\| q_{\Theta}^{\ast}(z)  - q_{\text{per}}(z)
    \right\|_{F}}{\left\|q_{\text{per}}(z)
    \right\|_{F}}.
\end{align}
This metric quantifies the discrepancy between the true and reconstructed scatterer, thereby serving as a direct measure of the reconstruction accuracy of the algorithm.


\subsection{$2$D numerical examples}


We demonstrate the reconstruction ability of the algorithm for the following scatterers: a square, a kite, an austria profile, a smiley face (Figure \ref{fig: 2D reconstruction phaseless (24)}), and a two-disk configuration (Figure \ref{fig:Kiterecon}) . The relative errors reported in Table \ref{table: 2D results phaseless (24)} are consistently low, indicative of accurate reconstructions across all test cases. The algorithm effectively determines the boundaries, locations, sizes, and coefficient values of each scatterer, despite the problem's non-linear and ill-posed nature. The numerical results demonstrate the method's versatility in recovering functions with a wide range of characteristics: it successfully reconstructs functions with convex compact support, such as the square, as well as non-convex geometries, such as the kite. Furthermore, it handles functions with sharp corners such as the square, and can recover disconnected domains containing holes, as illustrated by the austria profile example.

In particular, our algorithm successfully reconstructs the two-disk scatterer even though the distance between the two disks is less than $\lambda/2$, the physical resolution limit. The first disk is centered at $(0.35, 0)$, and the second disk is centered at $( - 0.3, 0)$, both with a radius of $0.3$. At wave number $k=7$, the physical resolution is $(2\pi/k)/2 = 0.4488$, which exceeds this gap, making the two disks particularly challenging to resolve. As seen in Figure \ref{fig:Kiterecon} the close proximity of the two disks means that the imaging function is not able to completely recover their shapes. Nevertheless, stage $2$ of our algorithm successfully distinguishes and recover both scatterers. This is unlike the case for the kite, where the imaging function provides clearer geometrical information  (Figure \ref{figI(z)}).

All four scatterers show a sharp initial decrease in relative error within the first $10,000$ to $30,000$ iterations, after which the errors stabilize and plateau for the remainder of the iterations (Figure \ref{fig:Update process and Loss curve (2D) - Kite}). The $70,000$ iterations appears to mark a transition point (represent by the dotted vertical line in Figure \ref{fig:Update process and Loss curve (2D) - Kite}), after which all errors remain essentially stable. The logarithmic loss curves for all scatterers decrease rapidly in the early iterations, consistent with the relative error behavior. The logarithmic curves show some oscillatory behavior in the initial iterations before smoothing out, indicating initial instability in the optimization that resolves as training progresses. 

Lastly, satisfactory reconstructions are retained even in the presence of boundary data contaminated with noise.

 \begin{figure}[H]
    \centering
    \begin{subfigure}[b]{0.25\textwidth}
        \includegraphics[width=\linewidth]{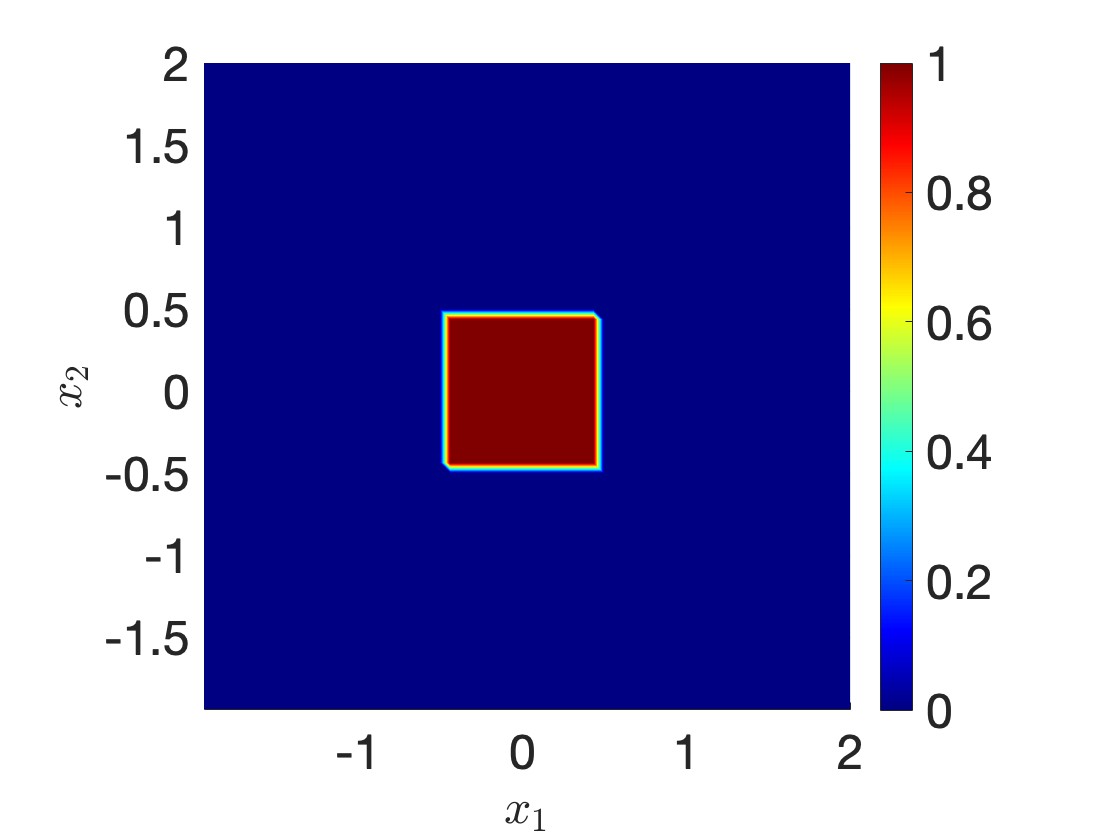}
    \end{subfigure}
    \hspace{-0.5cm}
    \begin{subfigure}[b]{0.25\textwidth}
        \includegraphics[width=\linewidth]{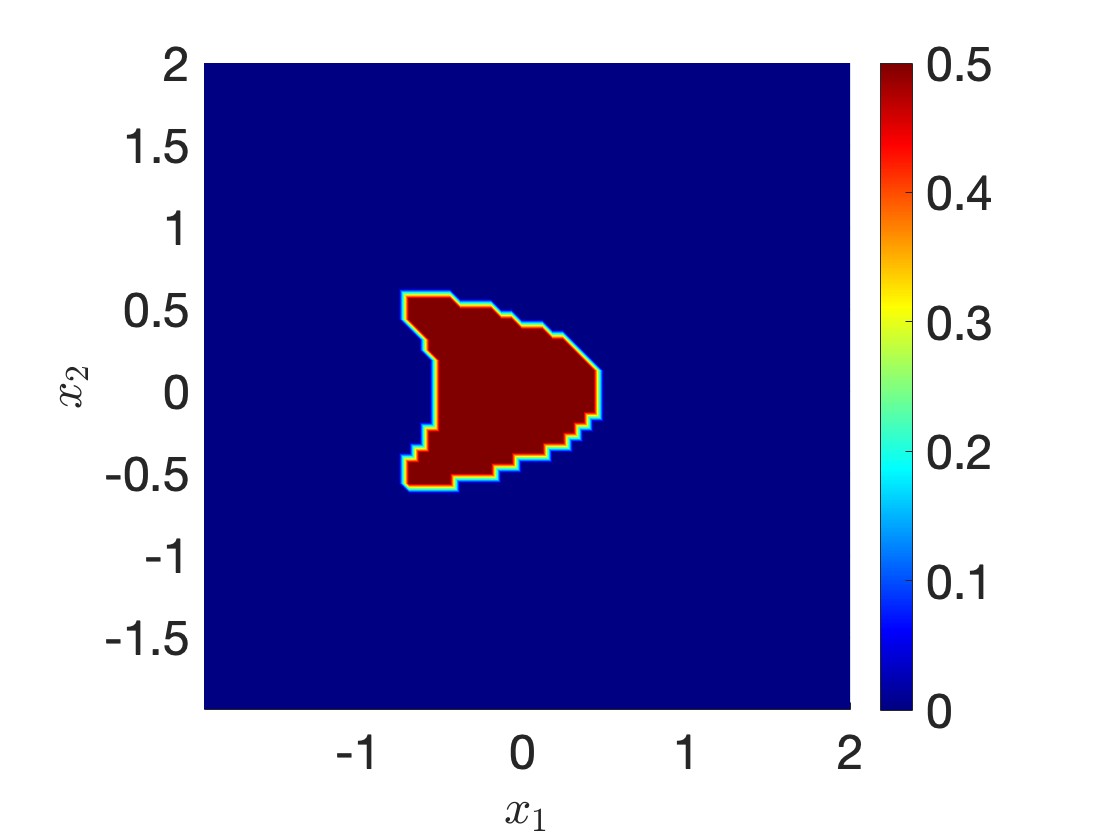}
    \end{subfigure}
    \hspace{-0.5cm}
    \begin{subfigure}[b]{0.25\textwidth}
        \includegraphics[width=\linewidth]{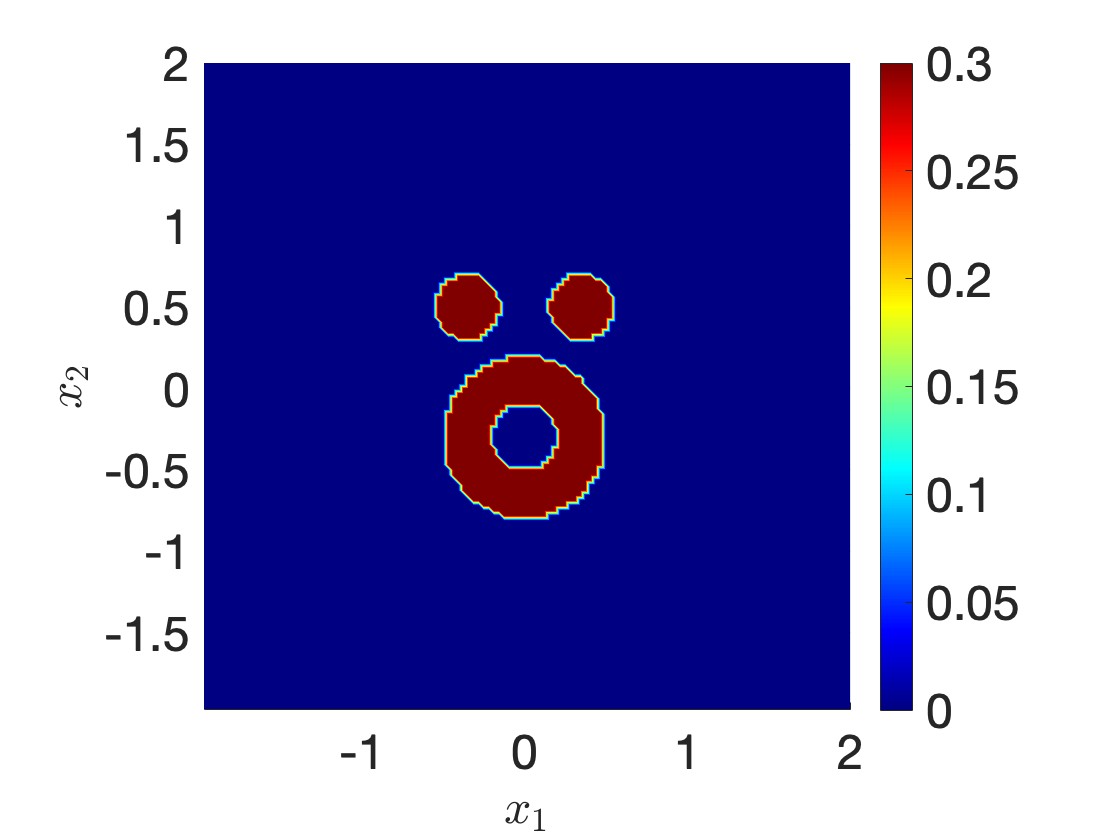}
    \end{subfigure}
    \hspace{-0.5cm}
    \begin{subfigure}[b]{0.25\textwidth}  \includegraphics[width=\linewidth]{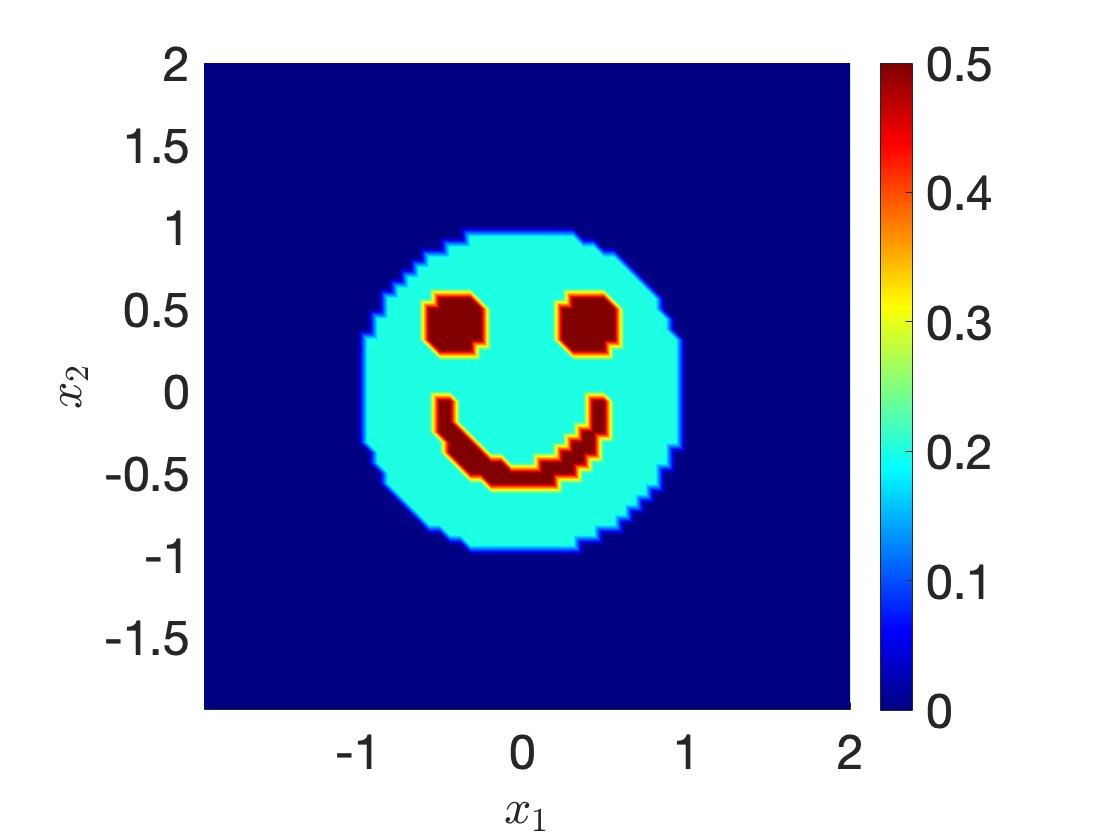}
    \end{subfigure}
    \hspace{-0.5cm}
    
    \medskip
    
    \centering
    \begin{subfigure}[b]{0.25\textwidth}
        \includegraphics[width=\linewidth]{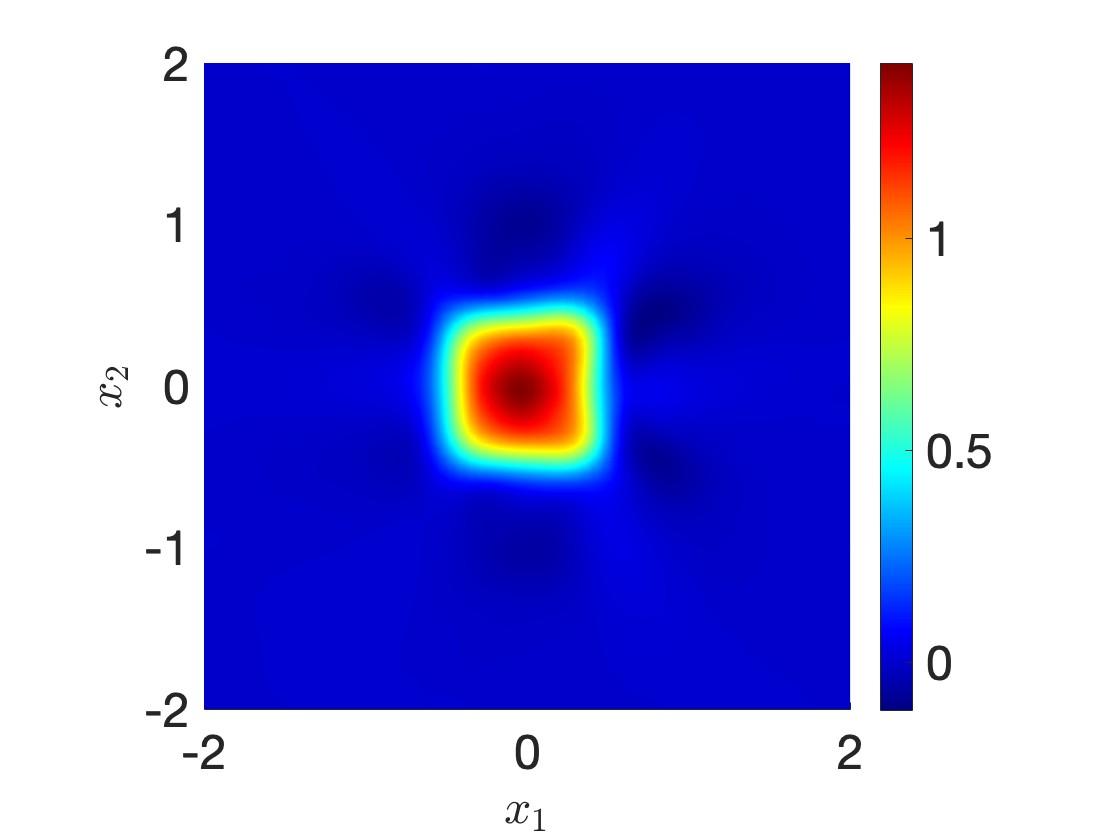}
    \end{subfigure}
    \hspace{-0.5cm}
    \begin{subfigure}[b]{0.25\textwidth}
        \includegraphics[width=\linewidth]{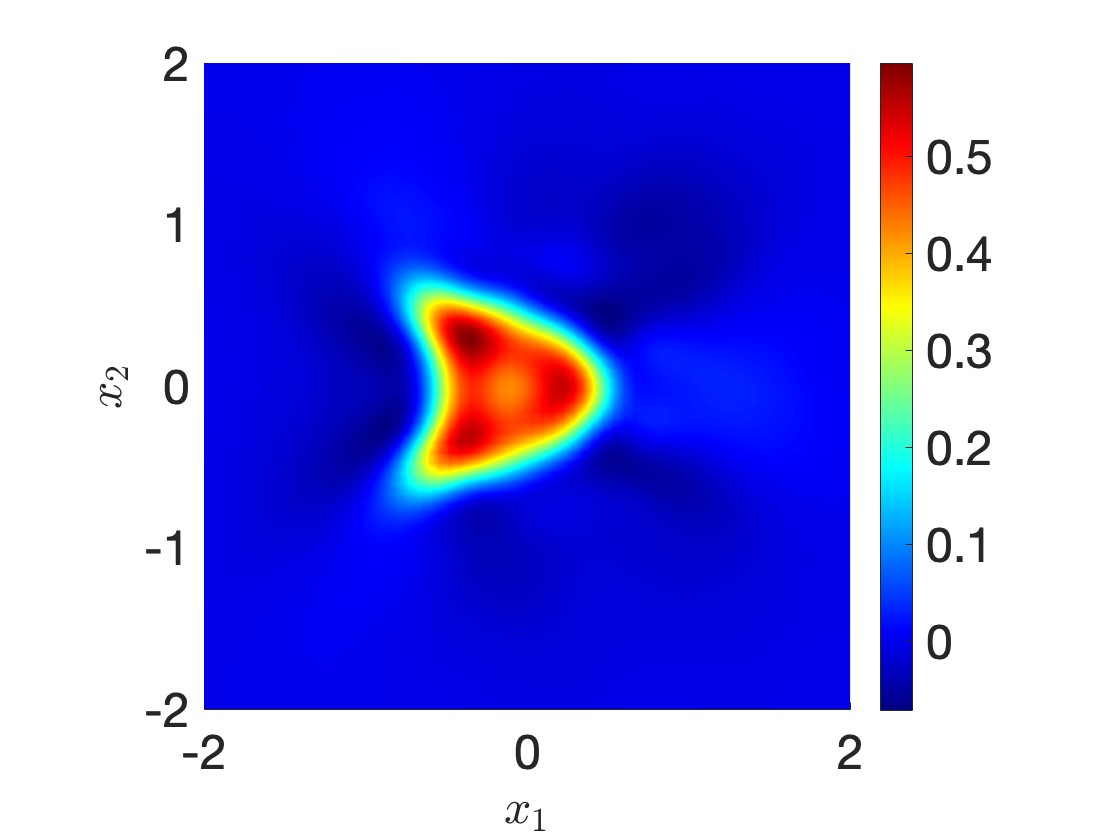}
    \end{subfigure}
    \hspace{-0.5cm}
    \begin{subfigure}[b]{0.25\textwidth}
        \includegraphics[width=\linewidth]{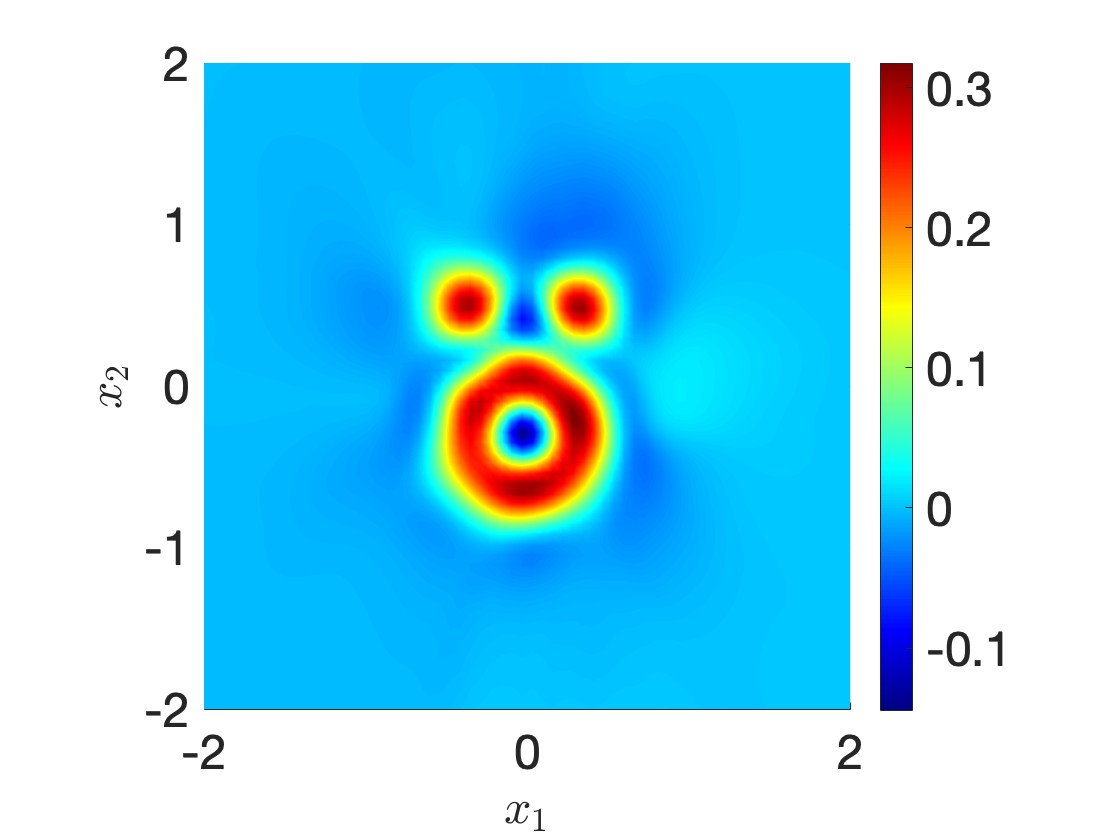}
    \end{subfigure}
    \hspace{-0.5cm}
    \begin{subfigure}[b]{0.25\textwidth}
\includegraphics[width=\linewidth]{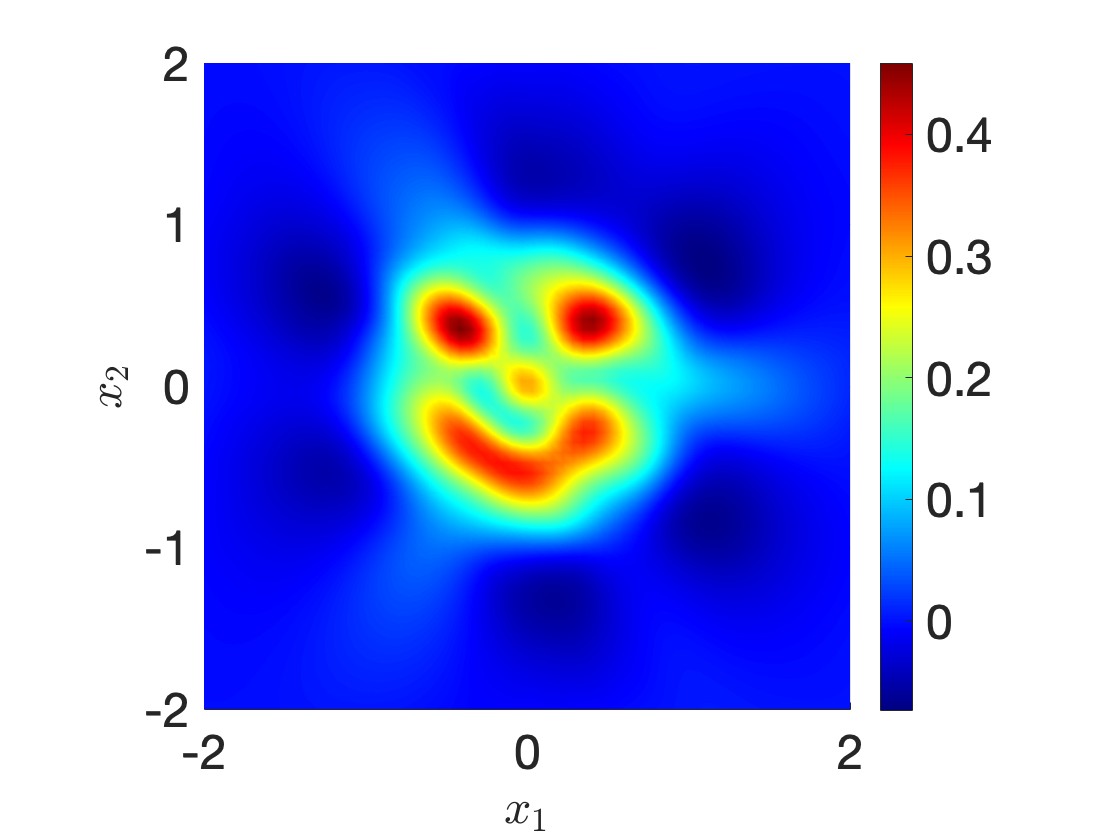}
    \end{subfigure}
    \hspace{-0.5cm}
    \caption{Reconstruction results for the $2$D case. Top row: true scatterers. Bottom row: reconstruction  by our algorithm.}
    \label{fig: 2D reconstruction phaseless (24)}
\end{figure}

\begin{figure}[H]
    \centering
    \begin{subfigure}[b]{0.3\textwidth}
    \includegraphics[width=\linewidth]{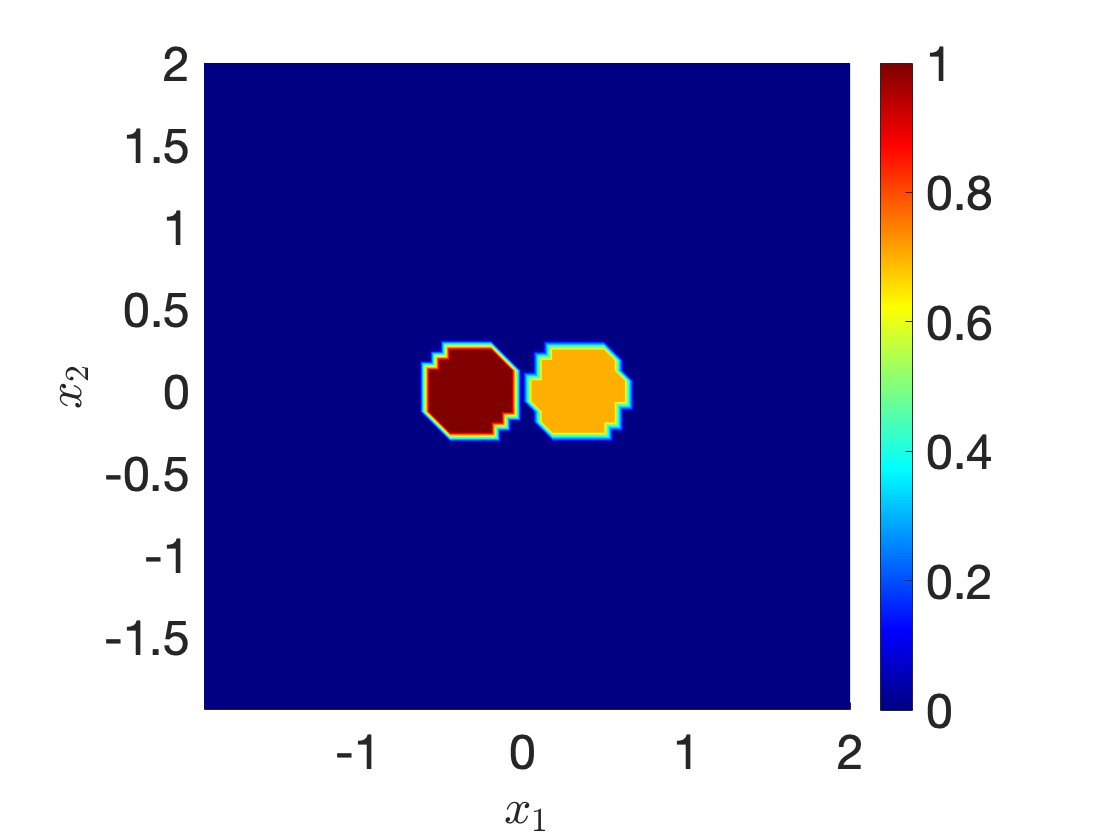}
    \end{subfigure}
    \hspace{-0.5cm}
    \begin{subfigure}[b]{0.3\textwidth}
    \includegraphics[width=\linewidth]{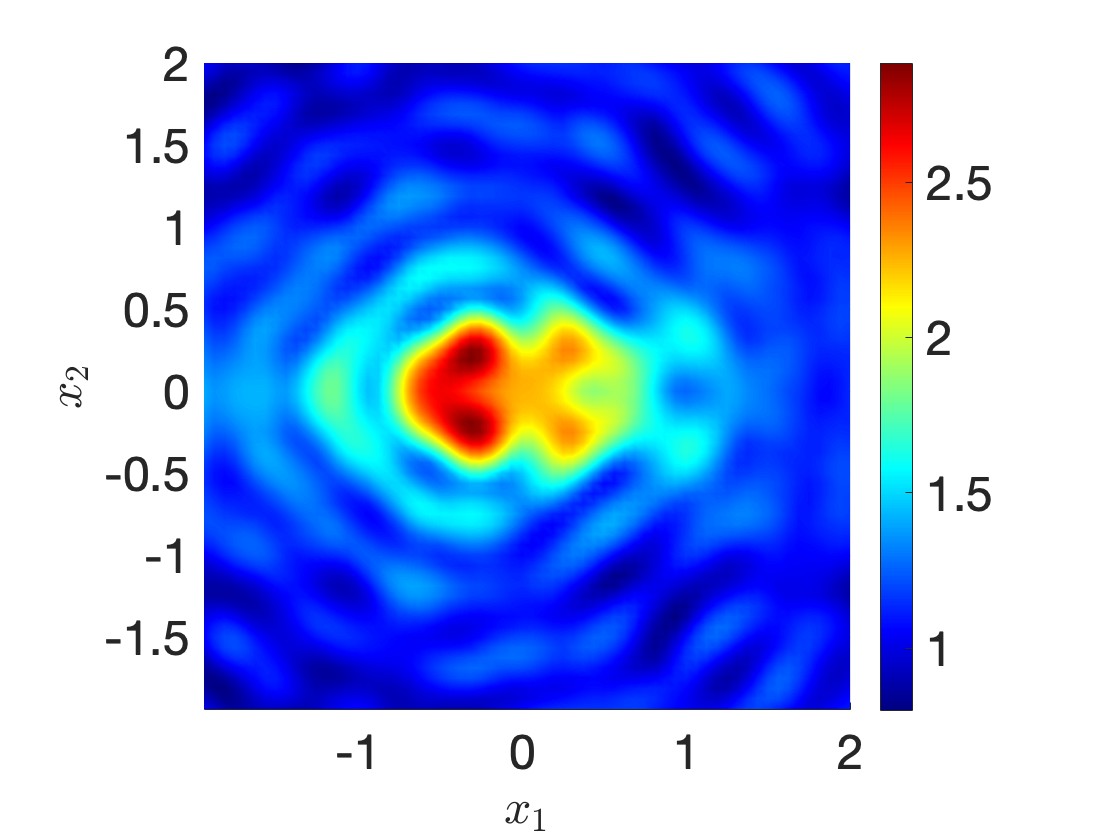}
    \end{subfigure}
    \hspace{-0.5cm}
    \begin{subfigure}[b]{0.3\textwidth}
\includegraphics[width=\linewidth]{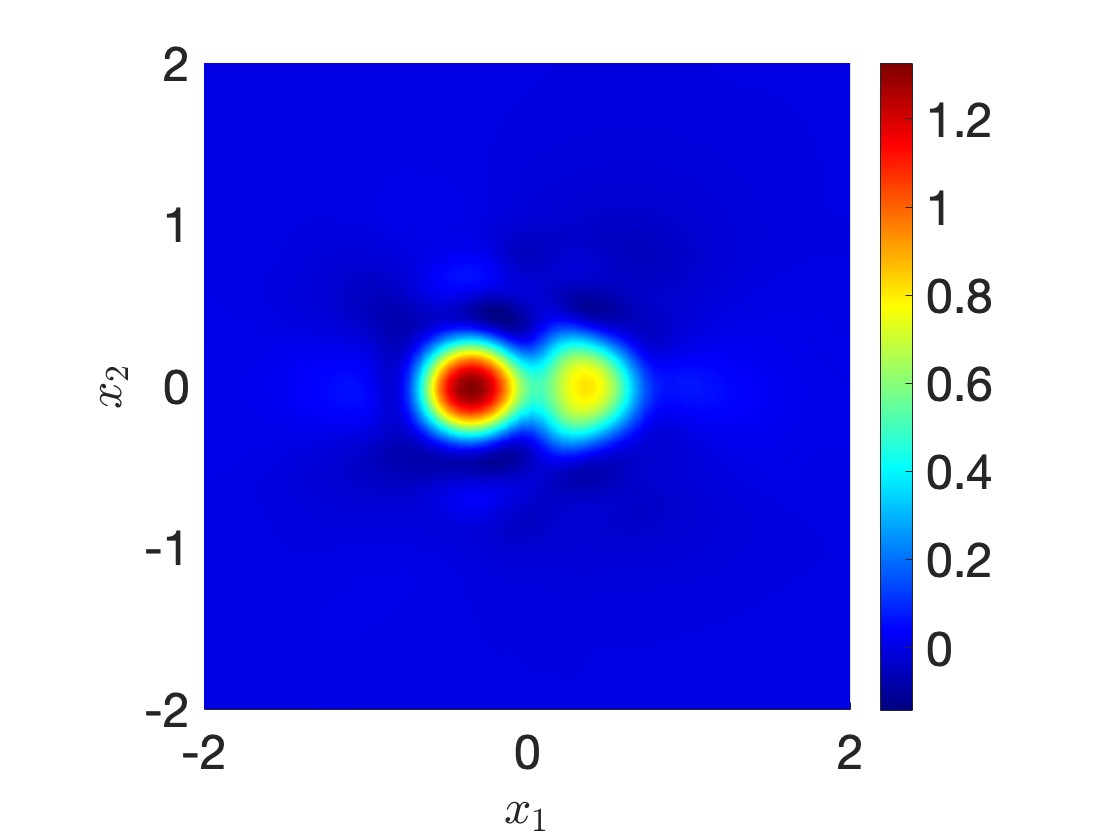}
    \end{subfigure}
    \caption{Reconstruction results for the inhomogeneous two-disk scatterer. Left: true scatterer. Center: $\sum_{m=1}^{10} \left|\mathcal{I}(z,d_m)\right|$. Right: reconstruction by our algorithm.}
    \label{fig:Kiterecon}
\end{figure}

\begin{table}[H]
    \centering
    \begin{tabular}{|c|c|c|c|}
    \hline
    Scatterer &  $\mathcal{E} \left[ q_{\Theta}^{\ast}\right]$ &  Computation time  \\
    \hline
    Square  & $36.4\% $  & 33 min 7 sec \\
    \hline
    Kite  & $35.9\% $  & 32 min 22 sec \\
    \hline
    Austria & $42.8\% $  & 43 min 2 sec \\
    \hline
    Smiley face  & $41.0\% $  & 32 min 58 sec \\
    \hline
    Two-disk  & $40.4\% $  & 41 min 46 sec \\
    \hline
    \end{tabular}
    \caption{Relative error and computation time for $2$D reconstructions}
    \label{table: 2D results phaseless (24)}
\end{table}

\begin{figure}[H]
    \centering
   
    \begin{subfigure}[b]{0.5\textwidth}
        \includegraphics[width=\linewidth]{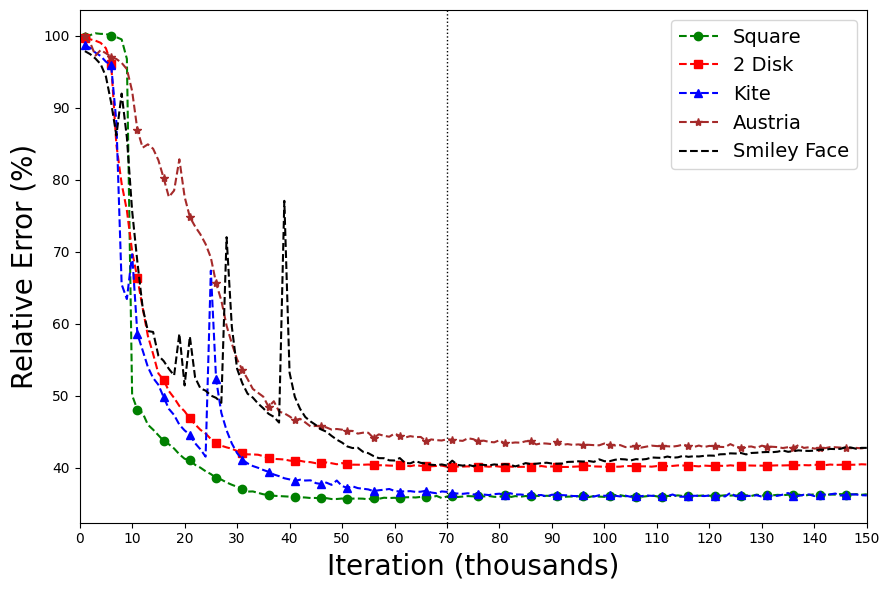}
        \caption{Relative error}
        \label{fig: Relative Error Plots V2}
    \end{subfigure}\begin{subfigure}[b]{0.5\textwidth}
        \includegraphics[width=\linewidth]{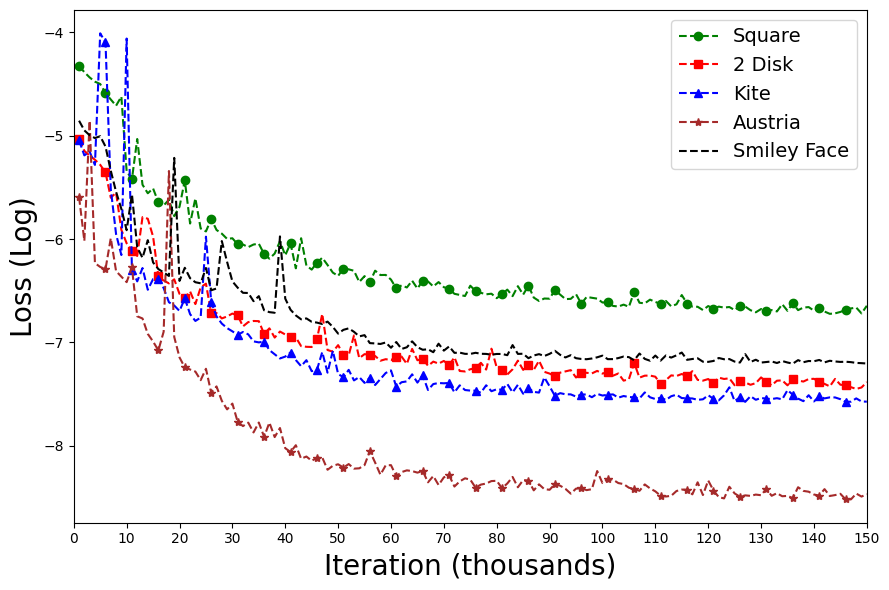}
        \caption{Log loss}
        \label{fig: Log loss Plots V2}
    \end{subfigure}
    \caption{(a) Relative errors between the true and predicted scatterers plotted against iteration number. (b) Evolution of the loss functions plotted against iteration number.} 
    \label{fig:Update process and Loss curve (2D) - Kite}
\end{figure}

\subsection{$3$D numerical examples}


Consistent with the $2$D results, the algorithm yields accurate reconstructions for the following $3$D scatterers: a cube, a sphere, and a pyramid (Figure \ref{fig: 3D reconstruction Slices with phaseless boundary data (k=7,M=16) V2}). In Figures \ref{fig: 3D reconstruction Slices with phaseless boundary data (k=7,M=16) V2}, \ref{fig: 3D reconstruction phaseless V3}, the computational domain is $[-2,2]^{3}$, while the results are displayed over the subdomain $[-1,1]^{3}$. For visualization purposes, the iso-surface plots are rendered at an iso-value of $30\%$ of the peak reconstruction value (Figure \ref{fig: 3D reconstruction isosurface (k=7,M=16)}). 

The algorithm successfully resolves the boundaries, locations, sizes, and coefficient values of each scatterer, as numerically substantiated by the low relative errors reported in Table \ref{table: 3D results phaseless (16)}.
Figure \ref{fig: 3D reconstruction isosurface (k=7,M=16)} highlights the algorithm's robustness across a range of geometries, demonstrating accurate recovery of both smooth scatterers (sphere) and non-smooth scatterers (cube, pyramid), as well as convex geometries. The algorithm further demonstrates versatility in handling more challenging configurations: Figure \ref{fig: 3D reconstruction phaseless V3} presents reconstructions of non-convex and disconnected scatterers, including a scatterer with an interior hole, which the algorithm resolves with high fidelity. The successful recovery of such topologically complex features underscores the method's ability to go beyond simple convex geometries and handle practically relevant, intricate scatterer shapes.

\begin{table}[H]
    \centering
    \begin{tabular}{|c|c|c|c|}
    \hline
    Scatterer &  $\mathcal{E} \left[ q_{\Theta}^{\ast}\right]$ &  Computation time  \\
    \hline
    Cube  & $45.7\% $  & 10 hrs 45 min 38 sec \\
    \hline
    Sphere  & $43.4\% $  & 10 hrs 42 min 44 sec \\
    \hline
    Pyramid & $50.1\% $  & 10 hrs 24 min 41 sec  \\
    \hline
    Cube with hole & $52.6\% $  & 8 hrs 12 min 21 sec  \\
    \hline
    Double sphere & $53.6\% $  & 10 hrs 22 min 10 sec  \\
    \hline
    \end{tabular}
    \caption{Relative error and computation time for $3$D reconstructions}
    \label{table: 3D results phaseless (16)}
\end{table}

\begin{figure}[H]
    \centering
    \begin{subfigure}[b]{0.35\textwidth}
        \includegraphics[width=\linewidth]{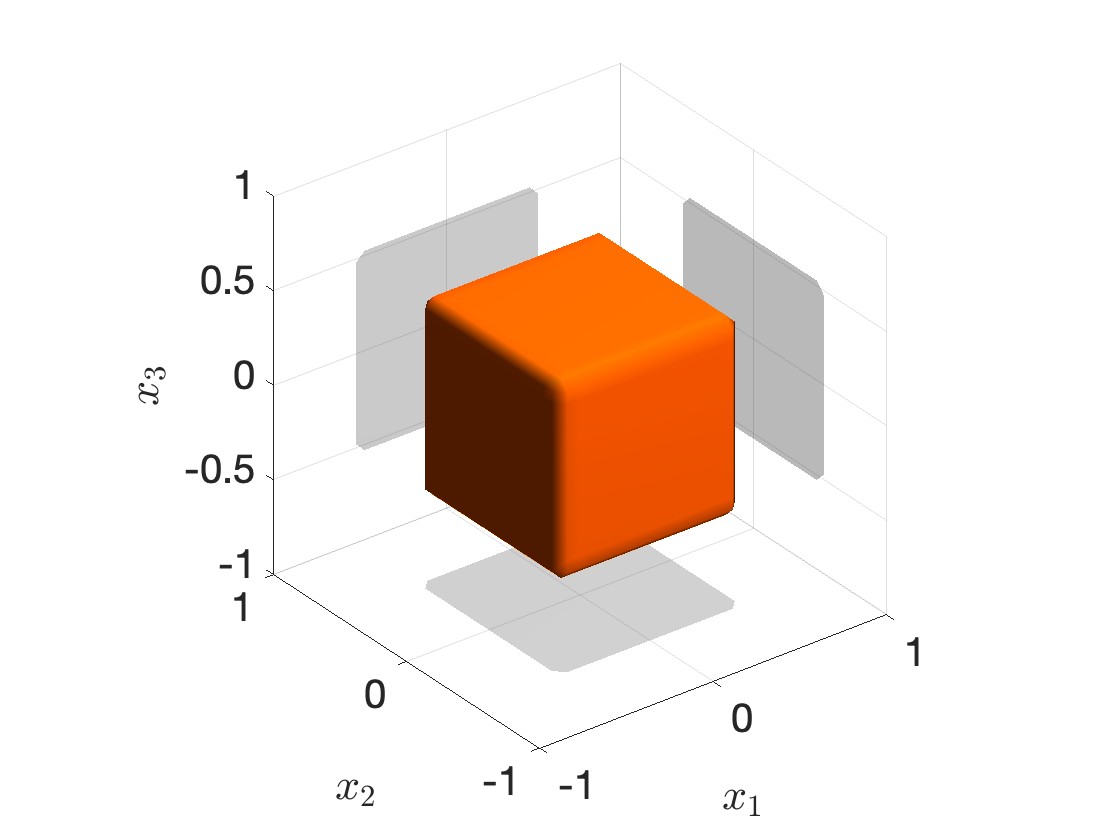}
    \end{subfigure}
    \hspace{-0.75cm}
    \begin{subfigure}[b]{0.35\textwidth}
        \includegraphics[width=\linewidth]{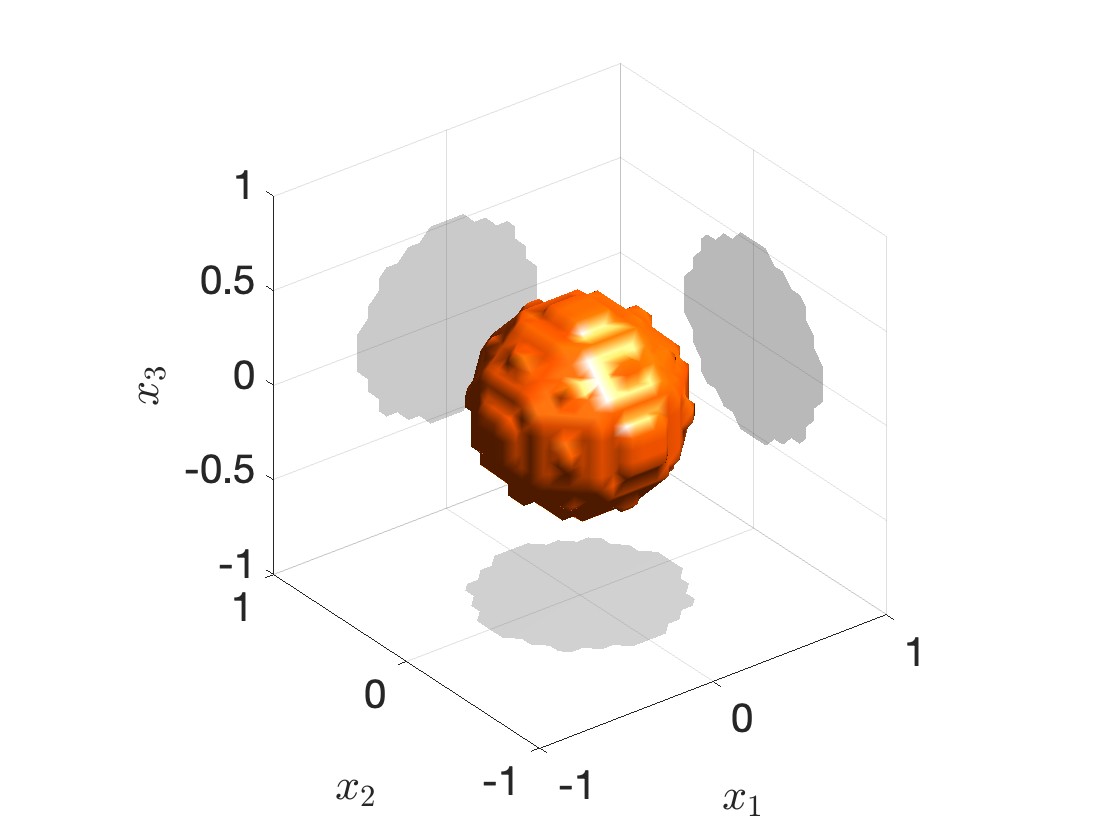}
    \end{subfigure}
    \hspace{-0.75cm}
    \begin{subfigure}[b]{0.35\textwidth}
        \includegraphics[width=\linewidth]{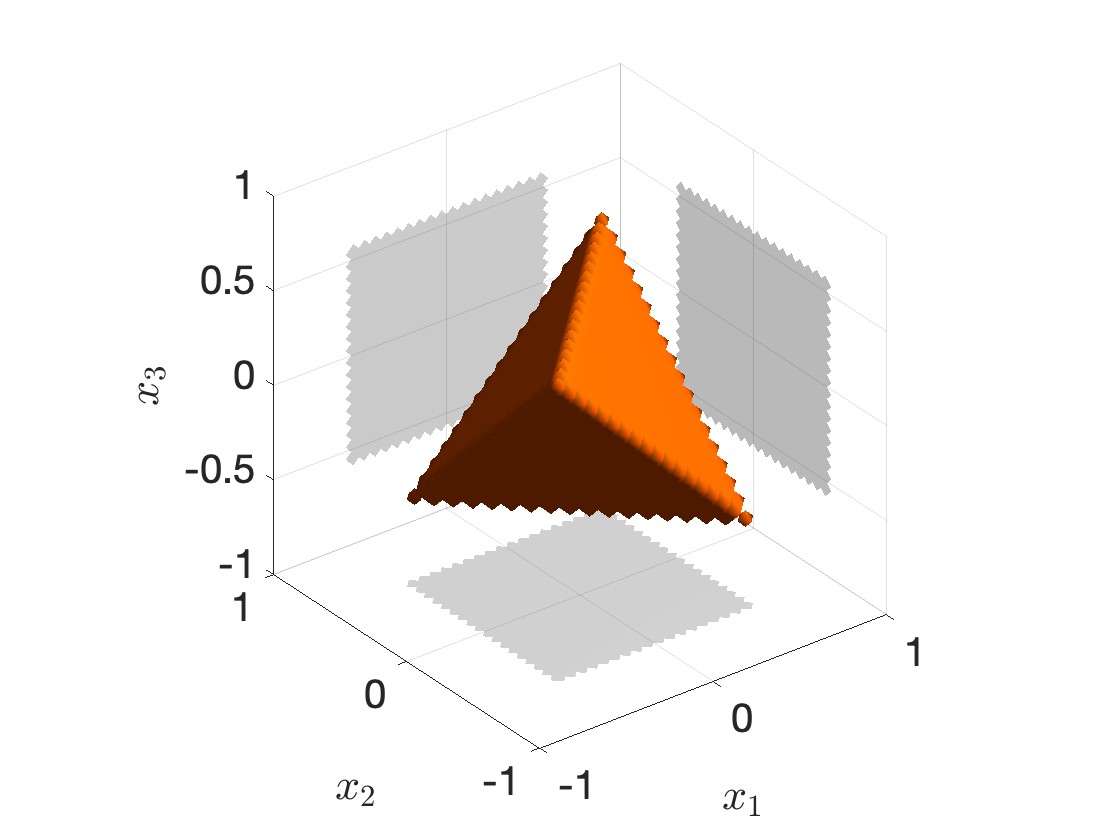}
    \end{subfigure}

    \begin{subfigure}[b]{0.35\textwidth}
        \includegraphics[width=\linewidth]{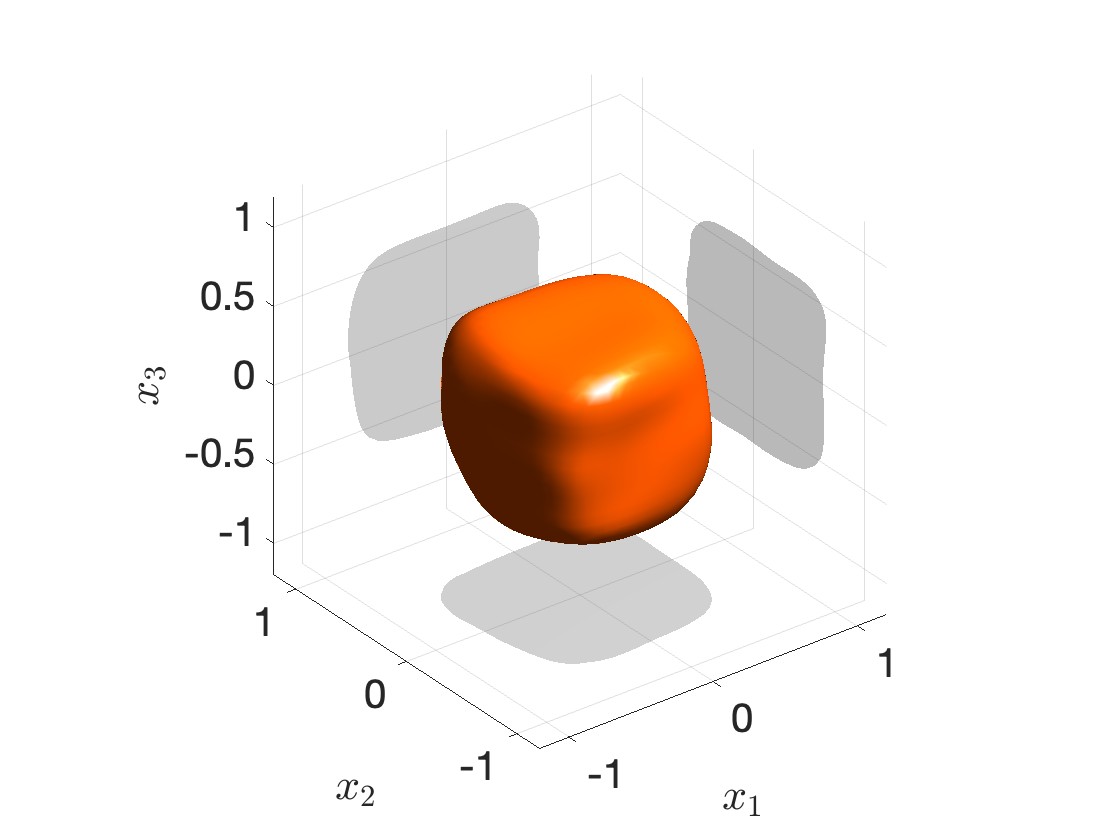}
    \end{subfigure}
    \hspace{-0.75cm}
    \begin{subfigure}[b]{0.35\textwidth}
        \includegraphics[width=\linewidth]{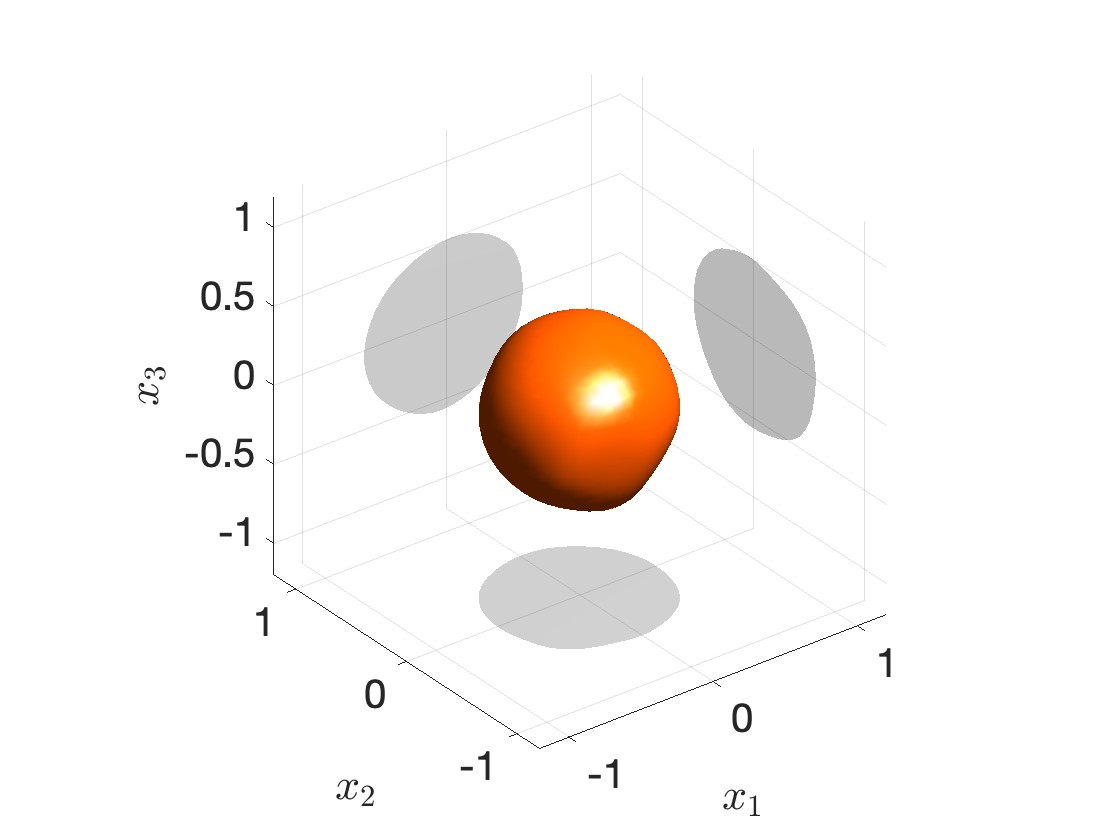}
    \end{subfigure}
    \hspace{-0.75cm}
    \begin{subfigure}[b]{0.35\textwidth}
        \includegraphics[width=\linewidth]{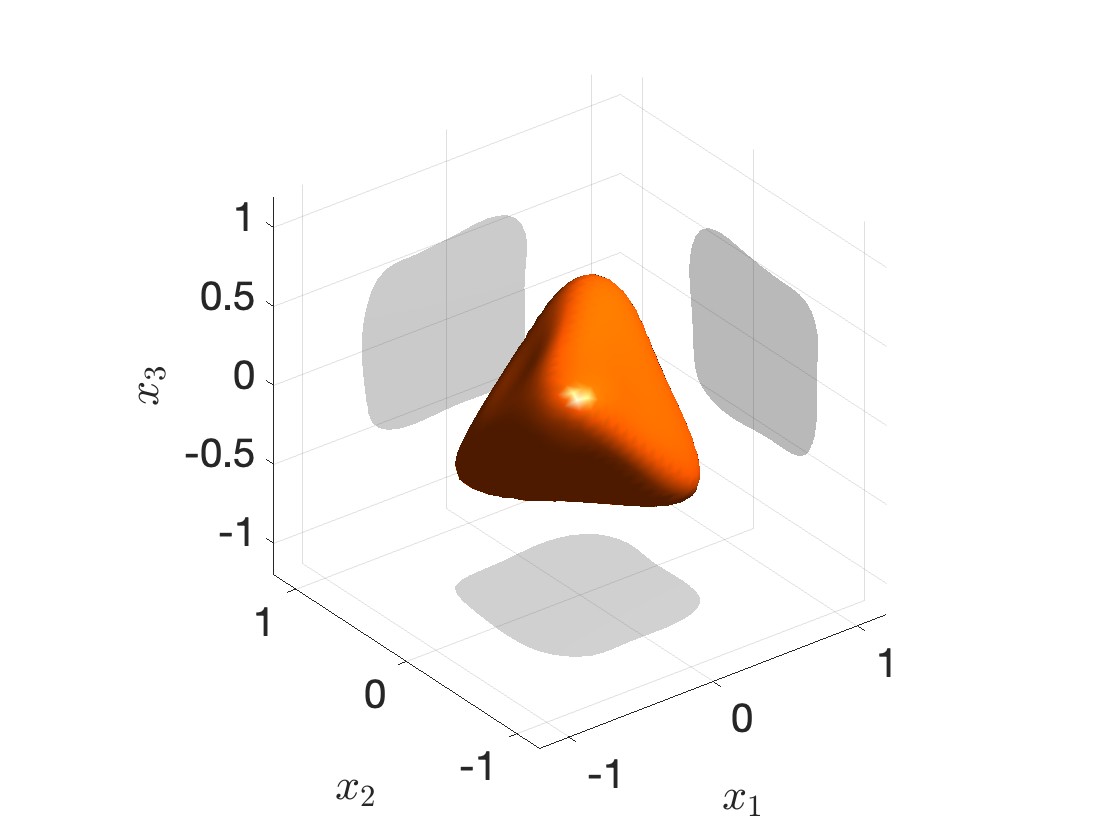}
    \end{subfigure}
    \caption{Iso-surface reconstruction results for the $3$D case. Top row: true scatterers. Bottom row: reconstruction by our algorithm.}
    \label{fig: 3D reconstruction isosurface (k=7,M=16)}
\end{figure}

\begin{figure}[H]
    \centering
     \begin{subfigure}[b]{0.3\textwidth}
        \includegraphics[width=\linewidth]{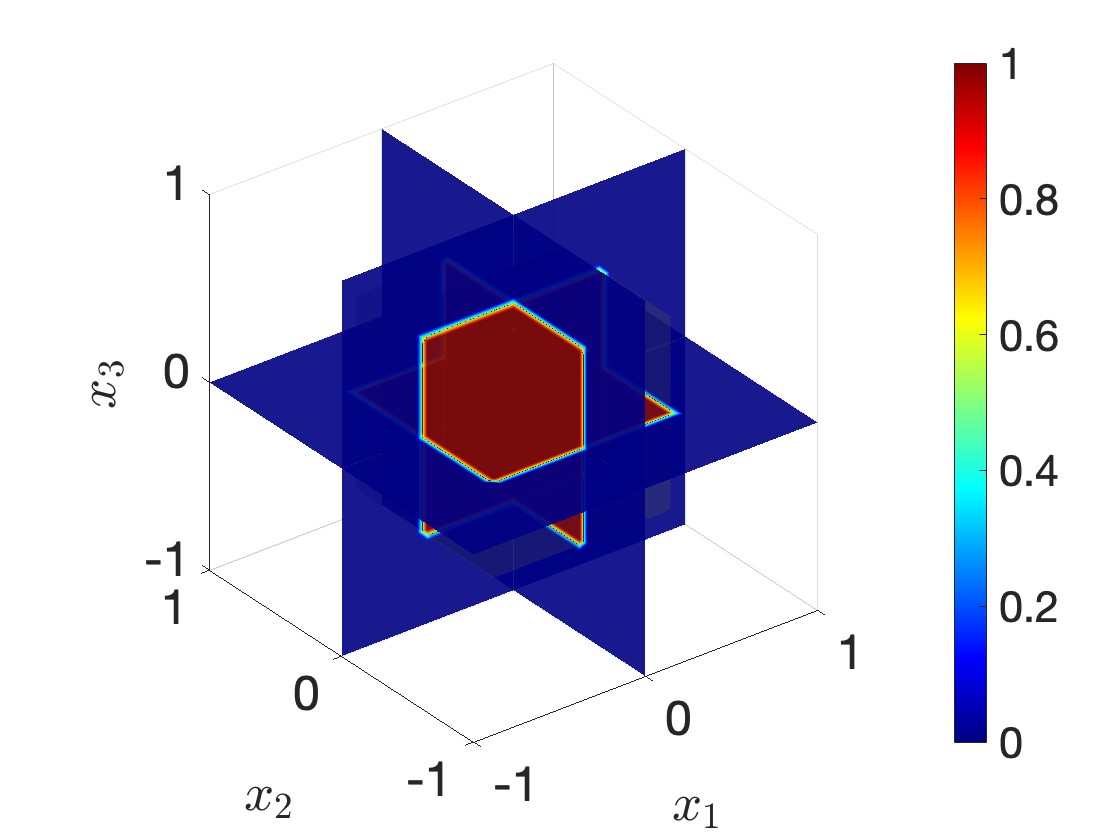}
    \end{subfigure}
    \hspace{-0.5cm}
    \begin{subfigure}[b]{0.3\textwidth}
        \includegraphics[width=\linewidth]{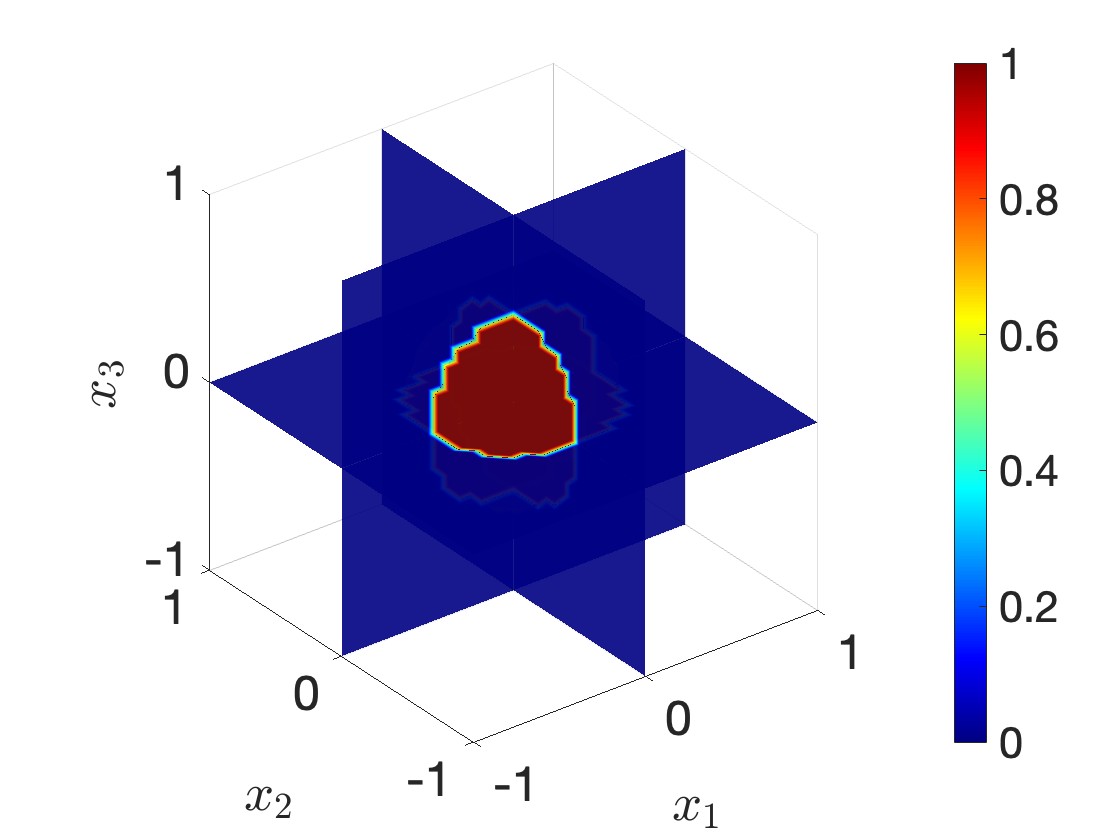}
    \end{subfigure}
\hspace{-0.5cm}
    \begin{subfigure}[b]{0.3\textwidth}
    \includegraphics[width=\linewidth]{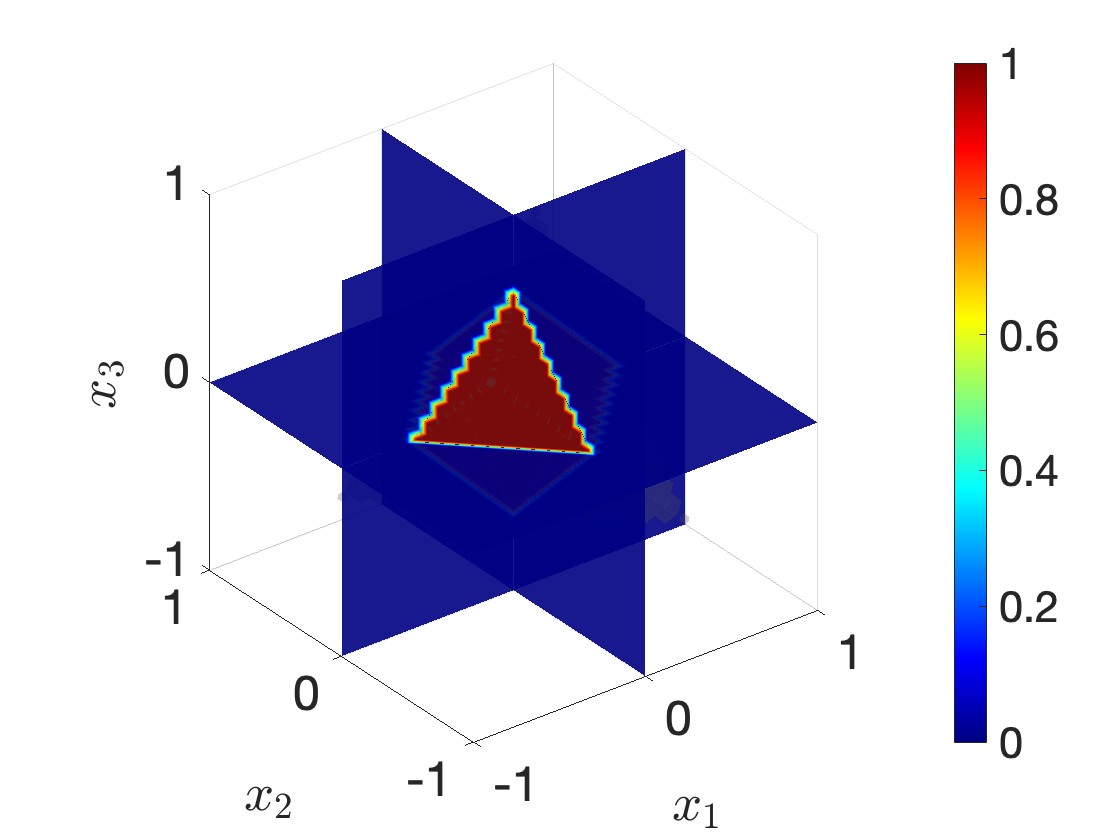}
    \end{subfigure}

    \begin{subfigure}[b]{0.3\textwidth}
    \includegraphics[width=\linewidth]{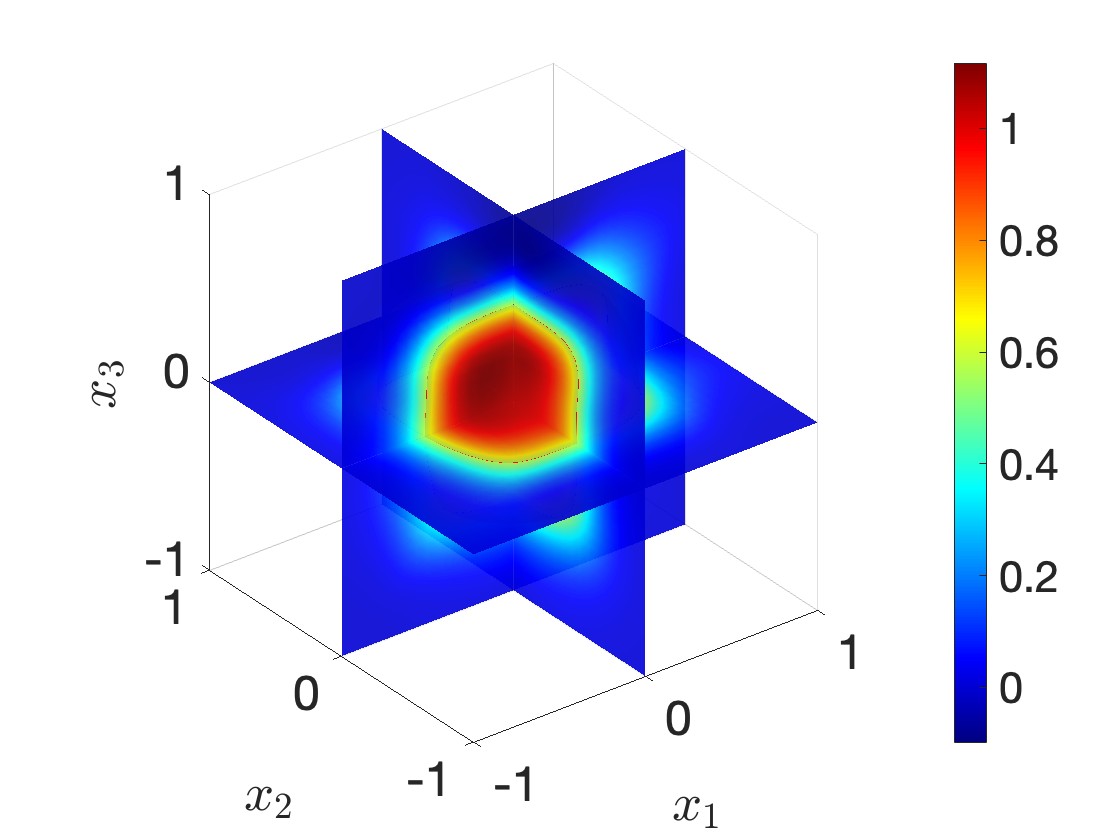}
    \end{subfigure}
    \hspace{-0.5cm}
    \begin{subfigure}[b]{0.3\textwidth}
    \includegraphics[width=\linewidth]{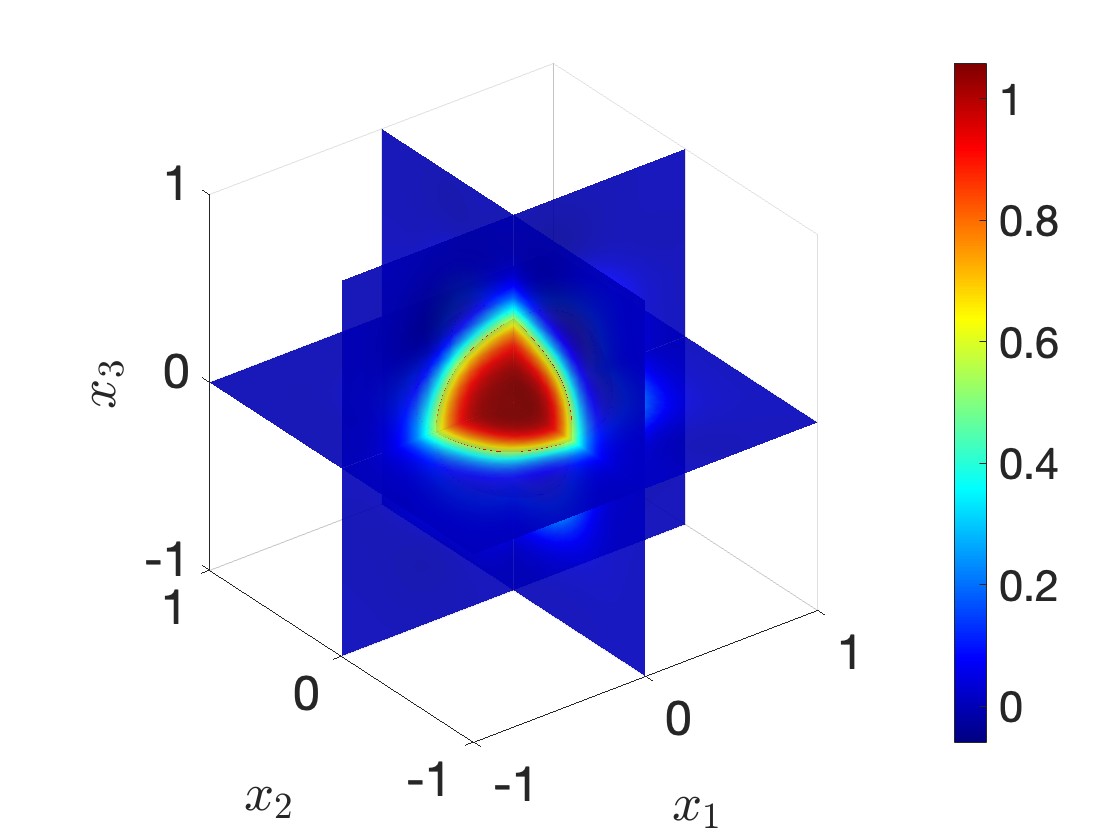}
    \end{subfigure}
    \hspace{-0.5cm}
    \begin{subfigure}[b]{0.3\textwidth}
    \includegraphics[width=\linewidth]{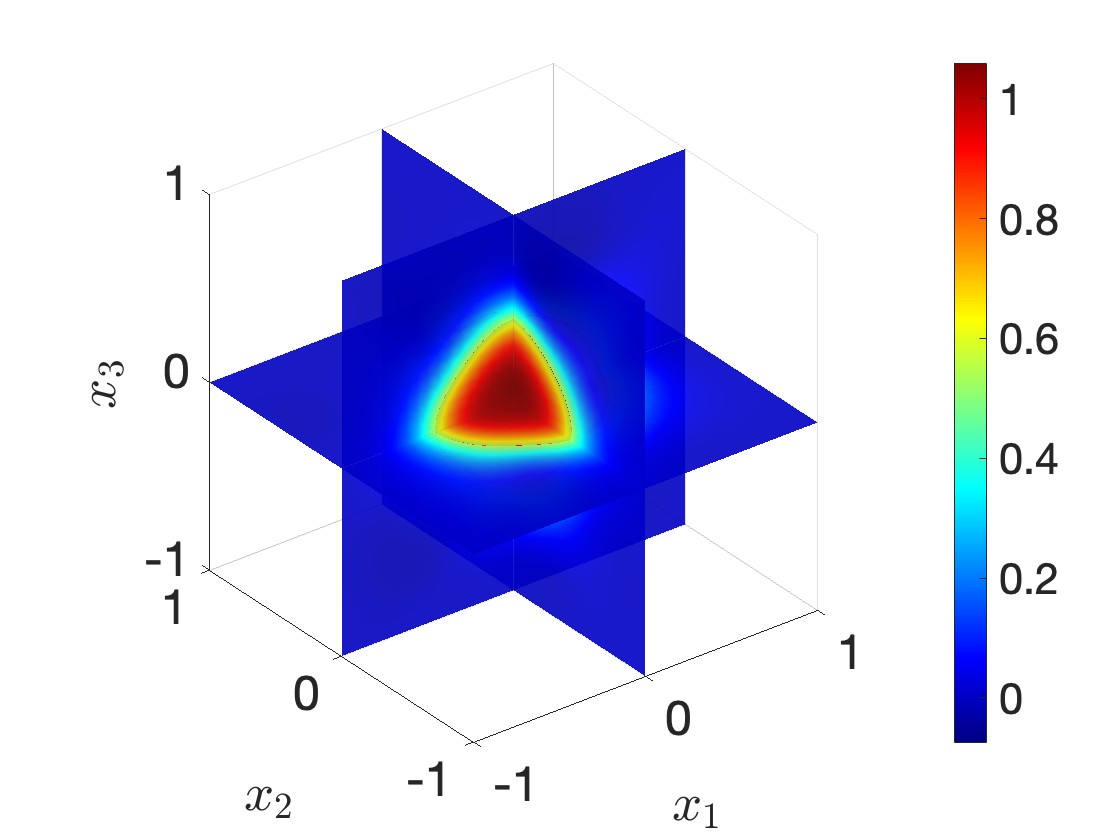}
    \end{subfigure}
    \caption{Slices reconstruction results for the $3$D case with phaseless boundary data. Top row: true scatterers. Bottom row: reconstruction by our algorithm.}
    \label{fig: 3D reconstruction Slices with phaseless boundary data (k=7,M=16) V2}
\end{figure}

 \begin{figure}[H]
    \centering
    \begin{subfigure}[b]{0.25\textwidth}
        \includegraphics[width=\linewidth]{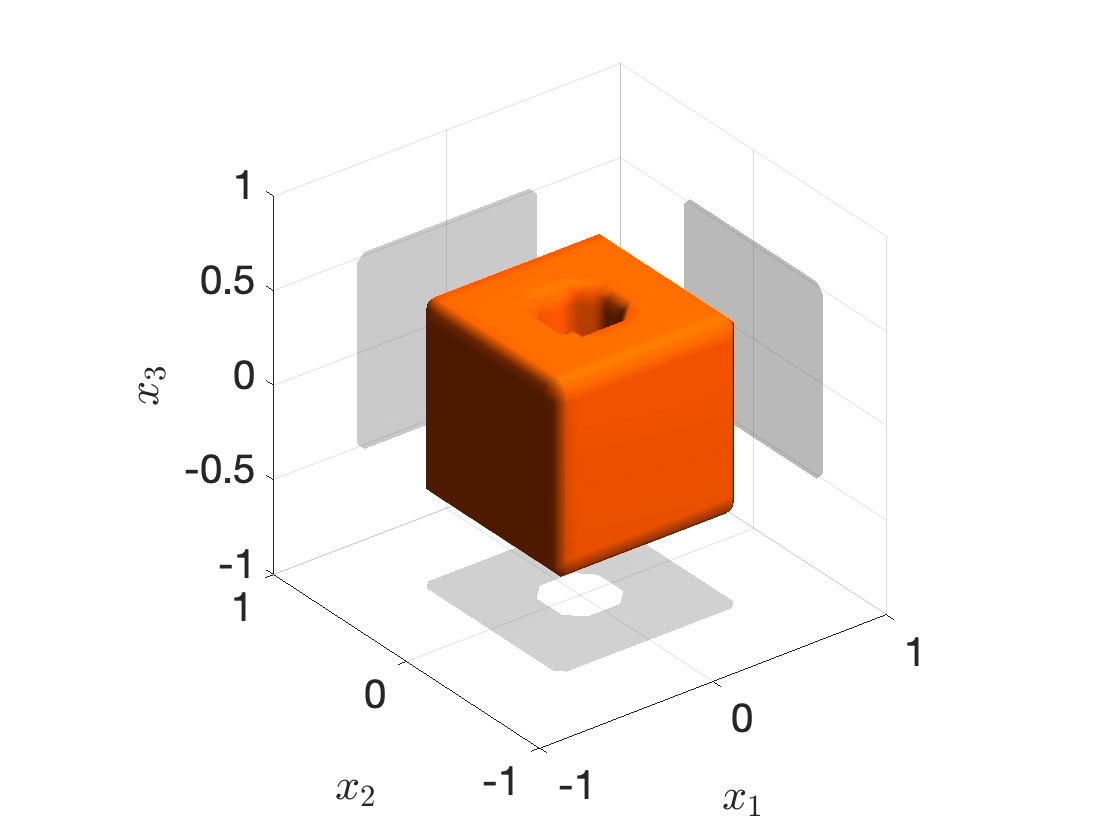}
    \end{subfigure}
    \hspace{-0.5cm}
    \begin{subfigure}[b]{0.25\textwidth}
        \includegraphics[width=\linewidth]{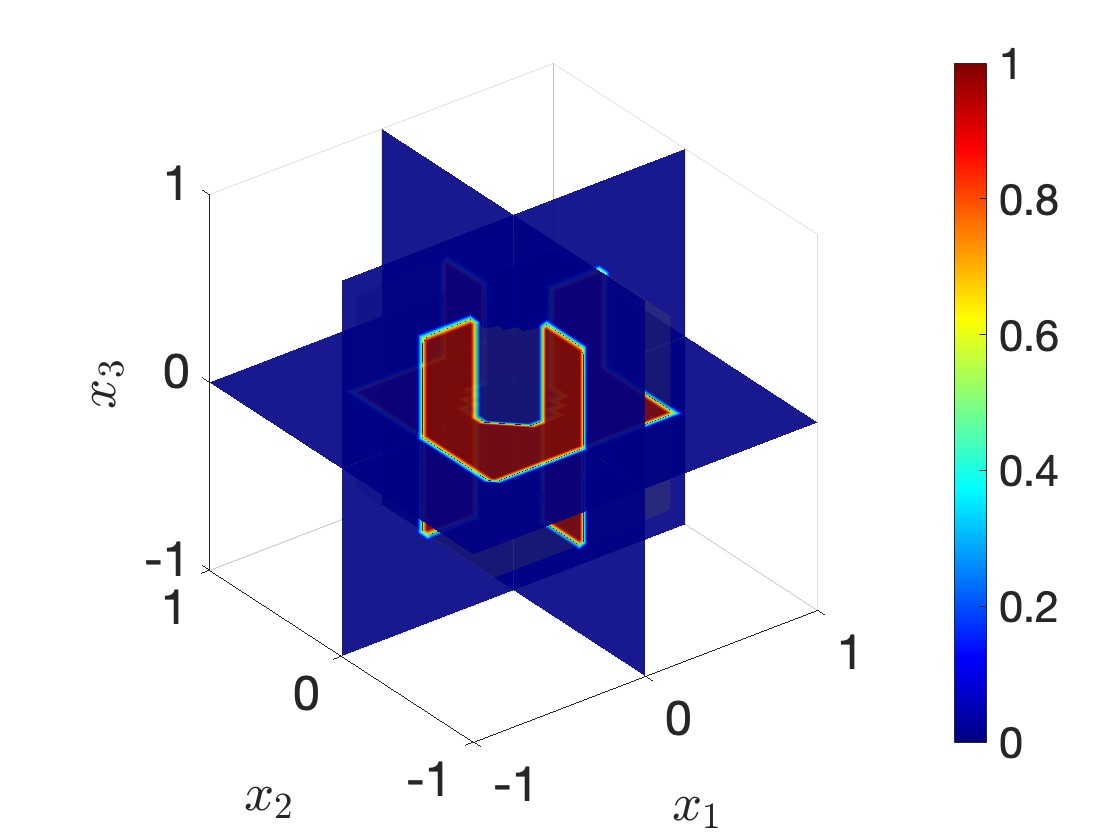}
    \end{subfigure}
    \hspace{-0.5cm}
    \begin{subfigure}[b]{0.25\textwidth}
        \includegraphics[width=\linewidth]{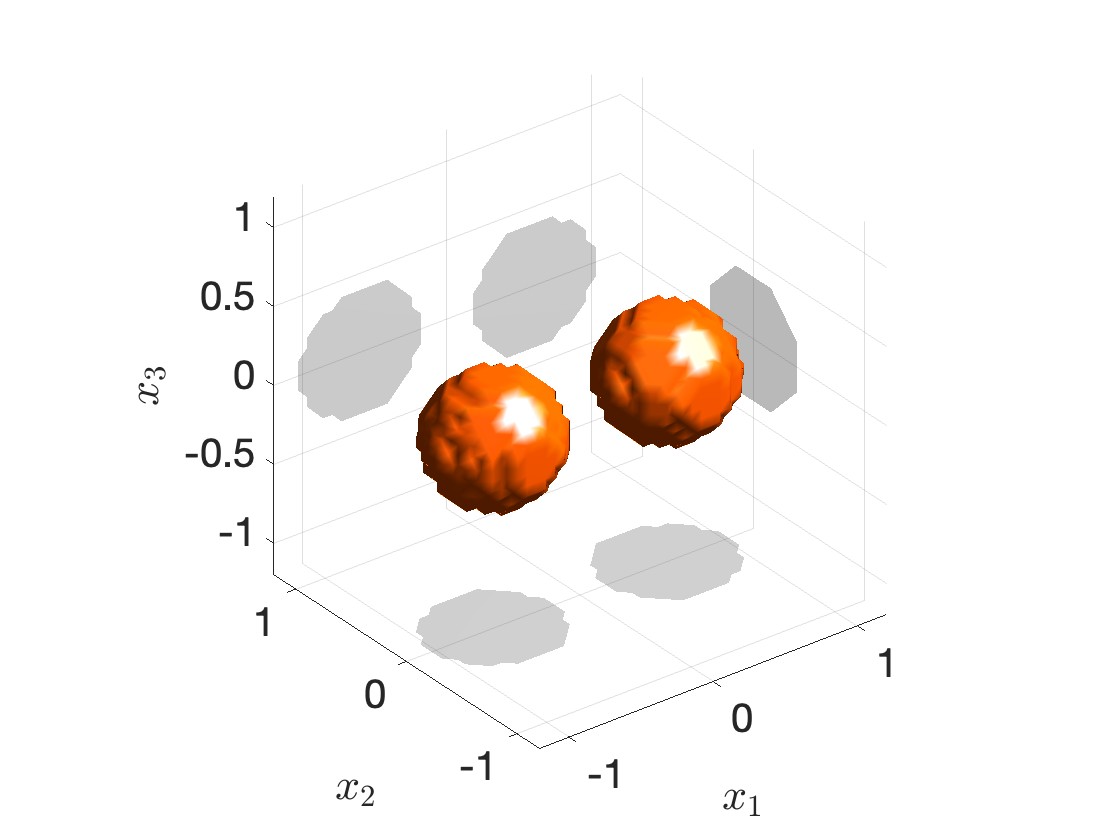}
    \end{subfigure}
    \hspace{-0.5cm}
    \begin{subfigure}[b]{0.25\textwidth}
        \includegraphics[width=\linewidth]{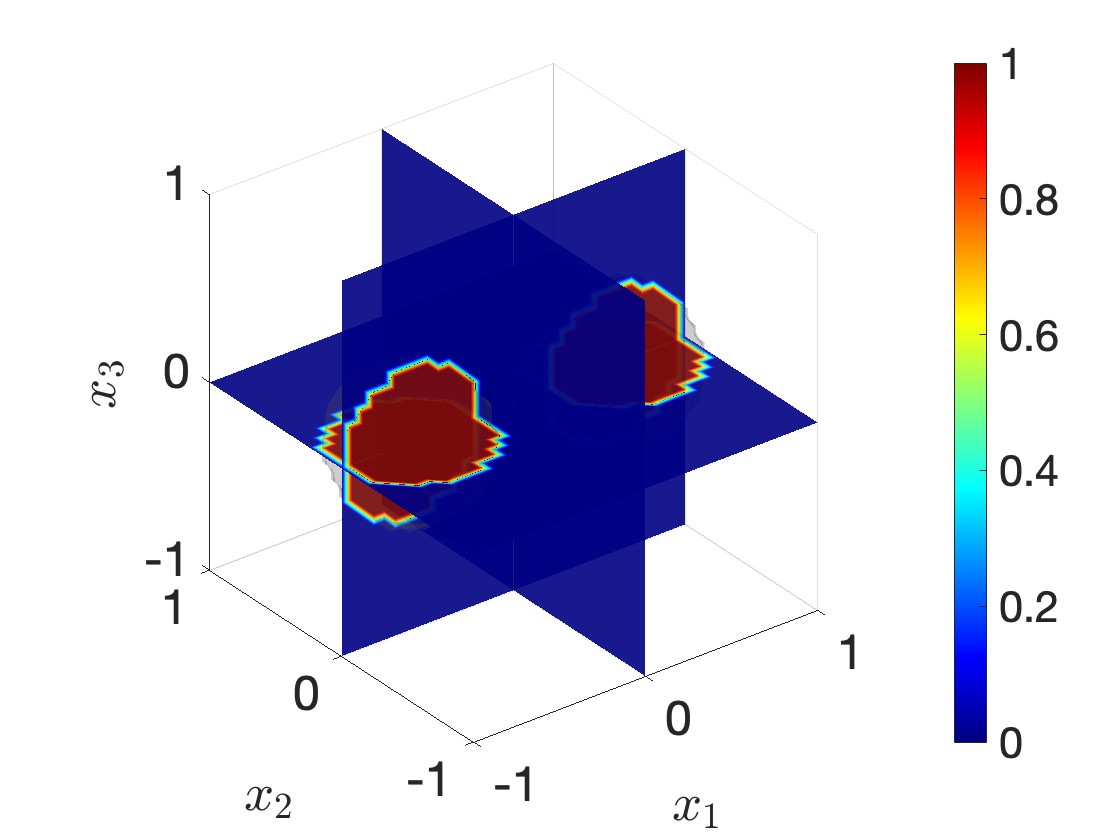}
    \end{subfigure}
    \hspace{-0.5cm}
    
    \medskip
    
    \centering
        \begin{subfigure}[b]{0.25\textwidth}
        \includegraphics[width=\linewidth]{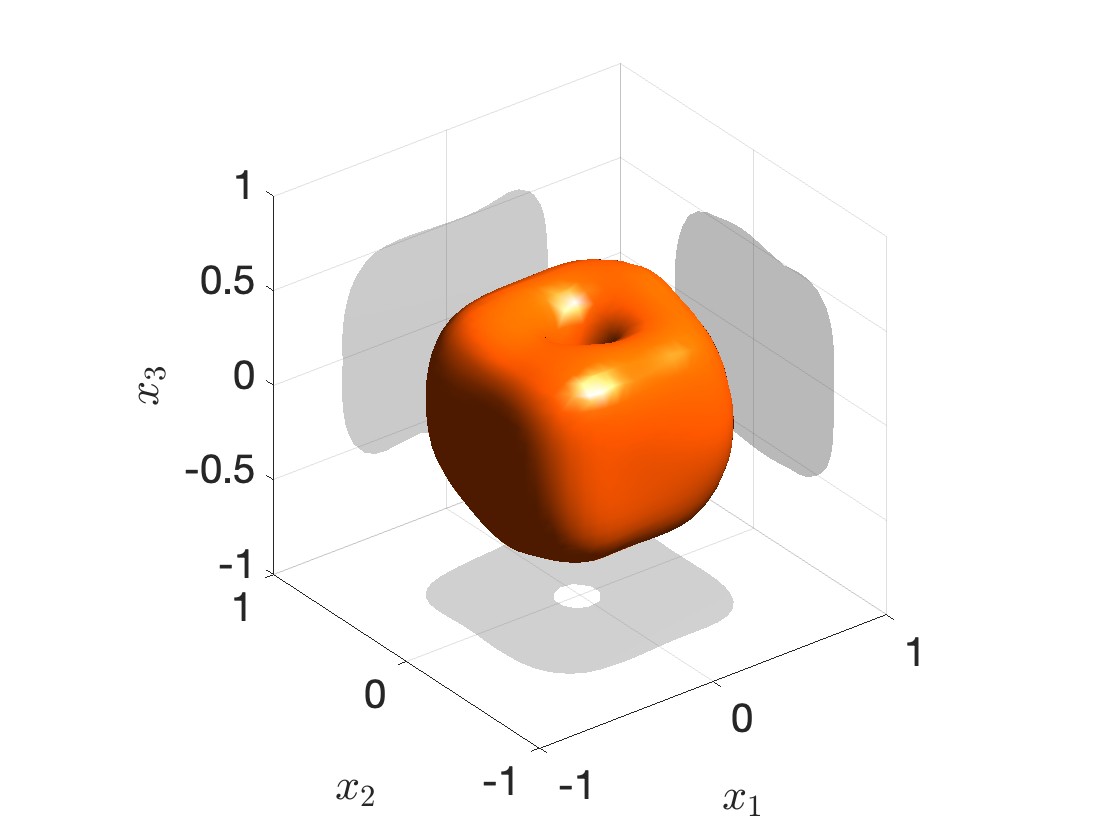}
    \end{subfigure}
    \hspace{-0.5cm}
    \begin{subfigure}[b]{0.25\textwidth}
        \includegraphics[width=\linewidth]{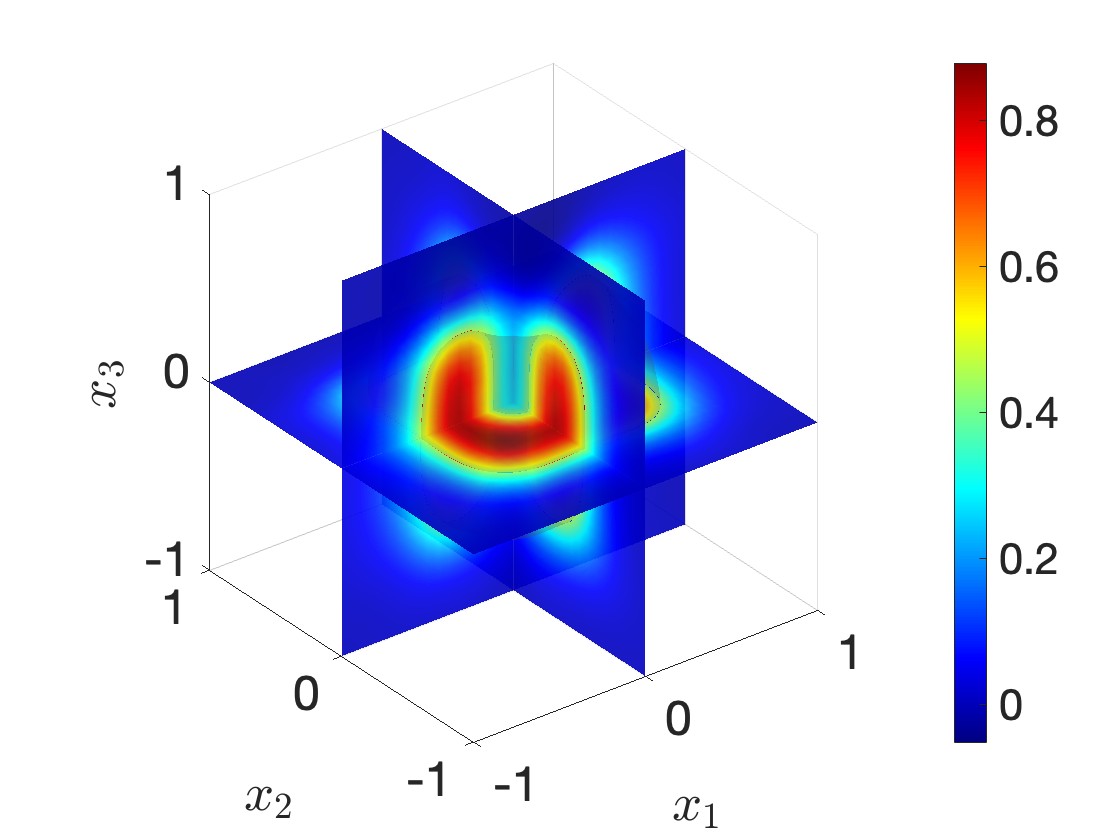}
    \end{subfigure}
    \hspace{-0.5cm}
    \begin{subfigure}[b]{0.25\textwidth}
        \includegraphics[width=\linewidth]{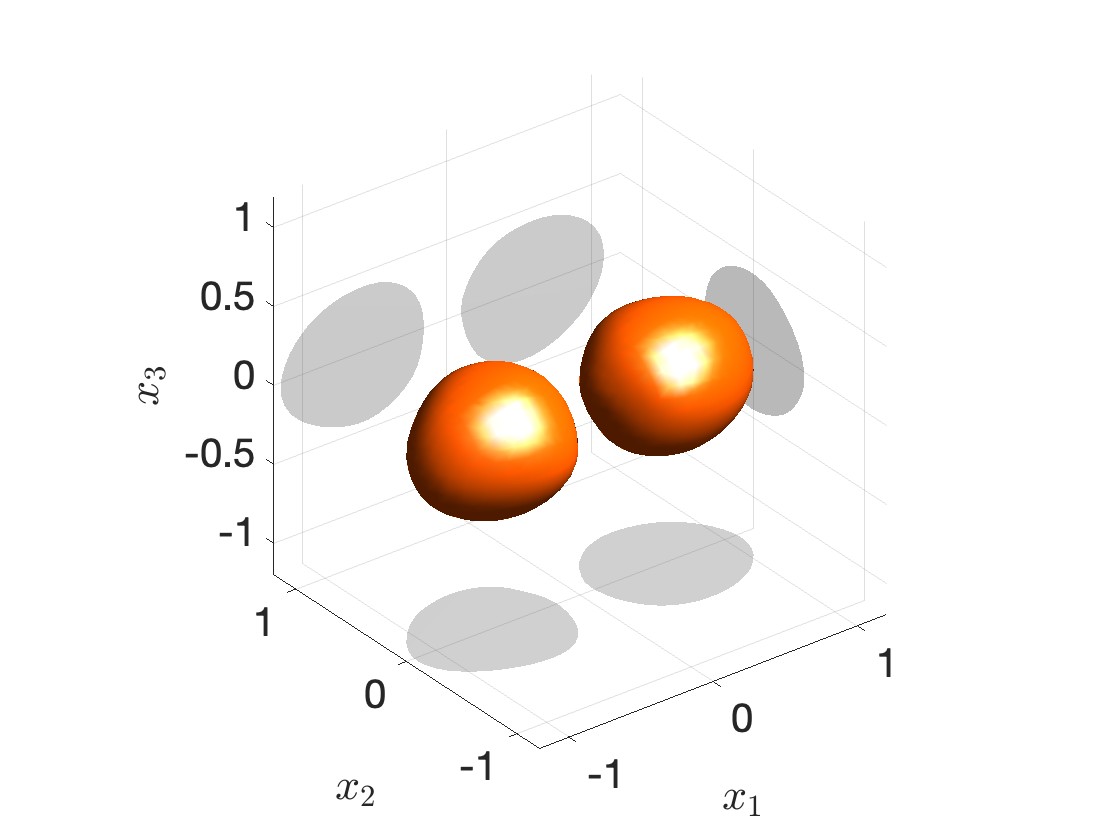}
    \end{subfigure}
    \hspace{-0.5cm}
    \begin{subfigure}[b]{0.25\textwidth}
        \includegraphics[width=\linewidth]{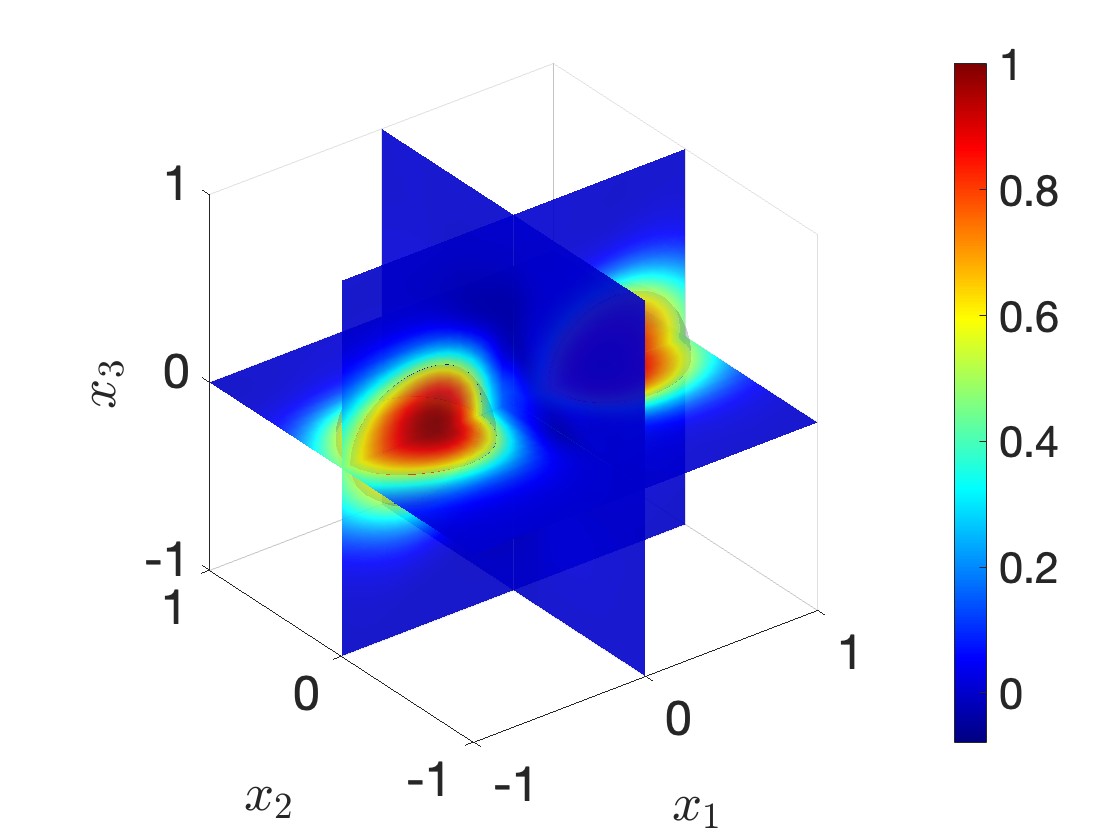}
    \end{subfigure}
    \hspace{-0.5cm}
    \caption{Reconstruction results for more complex geometries in $3$D case.}
    \label{fig: 3D reconstruction phaseless V3}
\end{figure}

\section{Conclusion}

We proposed a novel unsupervised deep learning algorithm for solving two- and three-dimensional inverse scattering problems using phaseless boundary measurement data. First, we constructed a system of approximate model equations. We verified that this imaging function provides reliable initial information about the location and geometry of the unknown scatterers. To improve computational efficiency and eliminate the Green’s function singularity, we reformulated the system in the Fourier domain, where we also establish a theoretical lower bound for Fourier series truncation, allowing fast computation while maintaining the accuracy. The resulting model equations then guide the concurrent training of two neural networks, which independently approximate the contrast source function and the unknown scatterer. Numerical experiments demonstrate that the proposed method achieves fast, accurate, and stable reconstructions, maintains robustness against noise. Future work will extend this framework to electromagnetic scattering problems.

\vspace{0.3 cm}

\noindent  \textbf{Acknowledgment.} The research is partially supported by NSF Grant DMS-2243854.

\printbibliography[heading=bibintoc,title={References}]

\end{document}

%% file: preamble.tex
\usepackage{amsmath, amsthm, amssymb, amsfonts}
\usepackage{amssymb}
\usepackage{bbm}
\usepackage{bm}
\usepackage{graphicx}
\usepackage{xcolor}
\usepackage{hyperref}
\usepackage{enumitem}
\usepackage{cleveref}
\usepackage{array}

\usepackage{amsthm}

\usepackage{multicol}
\usepackage[
  top=3cm,
  bottom=3cm,
  left=2cm,
  right=2cm
]{geometry}
\usepackage{lipsum}
\usepackage{mwe}

\usepackage{graphicx}
\usepackage{caption}
\usepackage{subcaption}
\usepackage{float}

\usepackage[linesnumbered,ruled,vlined]{algorithm2e}

\theoremstyle{plain}

\usepackage[
  backend=biber,
  style=numeric,
  giveninits=true,
  maxbibnames=99
]{biblatex}

\numberwithin{equation}{section}

\newtheorem{theorem}{Theorem}[section]

\theoremstyle{definition}

\newtheorem*{remark}{Remark}

\newcommand{\R}{\mathbb{R}}

\newcommand{\Z}{\mathbb{Z}}

\renewcommand{\bar}[1]{\overline{#1}}

\newcommand{\K}{\cal{K}}

\let\strokeL\L
\DeclareRobustCommand{\L}{\ifmmode\mathbf{L}\else\strokeL\fi}

\makeatletter
\DeclareFontFamily{OMX}{MnSymbolE}{}
\DeclareSymbolFont{MnLargeSymbols}{OMX}{MnSymbolE}{m}{n}
\SetSymbolFont{MnLargeSymbols}{bold}{OMX}{MnSymbolE}{b}{n}
\DeclareFontShape{OMX}{MnSymbolE}{m}{n}{
    <-6>  MnSymbolE5
   <6-7>  MnSymbolE6
   <7-8>  MnSymbolE7
   <8-9>  MnSymbolE8
   <9-10> MnSymbolE9
  <10-12> MnSymbolE10
  <12->   MnSymbolE12
}{}
\DeclareFontShape{OMX}{MnSymbolE}{b}{n}{
    <-6>  MnSymbolE-Bold5
   <6-7>  MnSymbolE-Bold6
   <7-8>  MnSymbolE-Bold7
   <8-9>  MnSymbolE-Bold8
   <9-10> MnSymbolE-Bold9
  <10-12> MnSymbolE-Bold10
  <12->   MnSymbolE-Bold12
}{}

\let\llangle\@undefined
\let\rrangle\@undefined
\DeclareMathDelimiter{\llangle}{\mathopen}%
                     {MnLargeSymbols}{'164}{MnLargeSymbols}{'164}
\DeclareMathDelimiter{\rrangle}{\mathclose}%
                     {MnLargeSymbols}{'171}{MnLargeSymbols}{'171}
\makeatother

\DeclareFontShape{OT1}{cmr}{bx}{sc}{<-> cmbcsc10}{}

%% file: references.bib
@article{Chen2019PhysicsinformedNN,
  title={Physics-informed neural networks for inverse problems in nano-optics and metamaterials.},
  author={Yuyao Chen and Lu Lu and George Em Karniadakis and Luca Dal Negro},
  journal={Opt. Express},
  year={2019},
  volume={28},
  number = {8},
  pages={
          11618-11633
        }
 % url={https://api.semanticscholar.org/CorpusID:208547648}
}

@article{Chen2020ARO,
  title={A Review of Deep Learning approaches for Inverse Scattering Problems (Invited Review)},
  author={Xudong Chen and Zhun Wei and Maokun Li and Paolo Rocca},
  journal={Prog. Electromagn. Res.},
  year={2020},
  volume = {167},
  pages= {67-81}
}

@article{Raissi2019PhysicsinformedNN,
  title={Physics-informed neural networks: A deep learning framework for solving forward and inverse problems involving nonlinear partial differential equations},
  author={Maziar Raissi and Paris Perdikaris and George Em Karniadakis},
  journal={J. Comput. Phys.},
  year={2019},
  volume={378},
  pages={686-707}
}

@article{Zhou2022ANN,
  title={A neural network warm-start approach for the inverse acoustic obstacle scattering problem},
  author={Zhou, Mo and Han, Jiequn and Rachh, Manas and Borges, Carlos},
  journal={J. Comput. Phys.},
  volume={490},
  pages={112341},
  year={2023},
  publisher={Elsevier}
}

@article{Gao2021OnAA,
  title={On an artificial neural network for inverse scattering problems},
  author={Yu Gao and Hongyu Liu and Xianchao Wang and Kai Zhang},
  journal={J. Comput. Phys.},
  year={2021},
  volume={448},
  pages={110771}
}

@inproceedings{Pokkunuru2023ImprovedTO,
  title={Improved Training of Physics-Informed Neural Networks Using Energy-Based Priors: a Study on Electrical Impedance Tomography},
  author={Akarsh Pokkunuru and Pedram Rooshenas and Thilo Strauss and Anuj Abhishek and Taufiquar R. Khan},
  booktitle={International Conference on Learning Representations},
  year={2023}
}

@article{Berg2001ContrastSI,
  title={Contrast Source Inversion Method: State of Art},
  author={Van den Berg, Peter M and Abubakar, Arafat},
  journal={Prog. Electromagn. Res.},
  year={2001},
  volume={34},
  pages={189-218}
}

@article{Cuomo2022ScientificML,
   title={Scientific machine learning through physics--informed neural networks: Where we are and what’s next},
  author={Cuomo, Salvatore and Di Cola, Vincenzo Schiano and Giampaolo, Fabio and Rozza, Gianluigi and Raissi, Maziar and Piccialli, Francesco},
  journal={Journal of Scientific Computing},
  volume={92},
  number={3},
  pages={88},
  year={2022},
  publisher={Springer}
}

@article{Zhang2023SolvingAI,
  title={Solving an inverse source problem by deep neural network method with convergence and error analysis},
  author={Hui Zhang and Jijun Liu},
  journal={Inverse Problems},
  year={2023},
  volume={39},
  pages={075013},
   number={7}
}

@incollection{vainikko2000fast,
  title={Fast solvers of the Lippmann-Schwinger equation},
  author={Vainikko, Gennadi},
  booktitle={Direct and inverse problems of mathematical physics},
  pages={423--440},
  year={2000},
  publisher={Springer}
}

@article{Le2022SamplingTM,
  title={Sampling type method combined with deep learning for inverse scattering with one incident wave},
  author={Thu Le and Dinh-Liem Nguyen and Vu Nguyen and Trung Truong},
  journal={Contemp. Math.},
  year={2023},
  volume={784},
  pages={63-80}
}

@article{Bulyshev2004ThreedimensionalVM,
  title={Three-dimensional vector microwave tomography: theory and computational experiments},
  author={Alexander E. Bulyshev and Alexandre E Souvorov and Serguei Y. Semenov and Vitaly G. Posukh and Yuri E. Sizov},
  journal={Inverse Problems},
  year={2004},
  volume={20},
   number={4},
  pages={1239-1259},
  url={}
}

@book{Persico2014IntroductionTG,
  title={Introduction to ground penetrating radar: inverse scattering and data processing},
  author={Persico, Raffaele},
  year={2014},
  publisher={John Wiley \& Sons}
}

@article{Klibanov2014PhaselessIS,
  title={Phaseless Inverse Scattering Problems in Three Dimensions},
  author={Michael V. Klibanov},
  journal={SIAM J. Appl. Math.},
  year={2014},
  volume={74},
  number = {2},
  pages={392-410},
  url={}
}

@article{Klibanov2015ReconstructionPF,
  title={Reconstruction Procedures for Two Inverse Scattering Problems Without the Phase Information},
  author={Michael V. Klibanov and Vladimir G. Romanov},
  journal={SIAM J. Appl. Math.},
  year={2015},
  volume={76},
  number={1},
  pages={178-196},
  url={}
}

@article{Novikov2015FormulasFP,
  title={Formulas for phase recovering from phaseless scattering data at fixed frequency},
  author={Roman G. Novikov},
  journal={Bull. Sci. Math.},
  year={2015},
   number={8},
  volume={139},
  pages={923-936},
  url={}
}

@article{Xu2017UniquenessII,
  title={Uniqueness in Inverse Scattering Problems with Phaseless Far-Field Data at a Fixed Frequency},
  author={Xiaoxu Xu and Bo Zhang and Haiwen Zhang},
  journal={SIAM J. Appl. Math.},
  year={2017},
  volume={78},
  number={3},
  pages={1737-1753},
  url={}
}

@article{Bao2013NumericalSO,
  title={Numerical solution of an inverse diffraction grating problem from phaseless data.},
  author={Gang Bao and Peijun Li and Junliang Lv},
  journal={J. Opt. Soc. Am. A},
  year={2013},
  volume={30},
  number={3},
  pages={
          293-9
        }
}

@article{Chen2015PhaselessIB,
  title={Phaseless imaging by reverse time migration: acoustic waves},
  author={Chen, Zhiming and Huang, Guanghui},
  journal={Numer. Math. Theor. Meth. Appl.},
  volume={10},
  number={1},
  pages={1--21},
  year={2017},
  publisher={Cambridge University Press}
}

@article{Zhang2018FastIO,
  title={Fast imaging of scattering obstacles from phaseless far-field measurements at a fixed frequency},
  author={Bo Zhang and Haiwen Zhang},
  journal={Inverse Problems},
  number={10},
  pages={104005},
  year={2018},
  volume={34},
  url={}
}

@article{Zhang2020AnAF,
  title={An Approximate Factorization Method for Inverse Acoustic Scattering with Phaseless Total-Field Data},
  author={Bo Zhang and Haiwen Zhang},
  journal={SIAM J. Appl. Math.},
  year={2020},
  volume={80},
  pages={2271-2298},
  url={}
}

@article{nguyen2024tnet,
  title={TNET: A model-constrained Tikhonov network approach for inverse problems},
  author={Nguyen, Hai V and Bui-Thanh, Tan},
  journal={SIAM J. Sci. Comput.},
  volume={46},
  number={1},
  pages={C77--C100},
  year={2024},
  publisher={SIAM}
}

@article{Nguyen2024TAENAM,
  title={TAEN: A Model-Constrained Tikhonov Autoencoder Network for Forward and Inverse Problems},
  author={Hai Van Nguyen and Tan Bui-Thanh and Clint Dawson},
  journal={Comput. Methods Appl. Mech. Eng.},
  volume = {446},
  pages = {118245},
  year = {2025}
}

@article{zhang2020unique,
  title={Unique determinations in inverse scattering problems with phaseless near-field measurements},
  author={Zhang, Deyue and Guo, Yukun and Sun, Fenglin and Liu, Hongyu},
  journal={Inverse Probl. Imaging},
  volume={14},
  number={3},
  year={2020}
}

@article{klibanov2016two,
  title={Two reconstruction procedures for a 3D phaseless inverse scattering problem for the generalized Helmholtz equation},
  author={Klibanov, Michael V and Romanov, Vladimir G},
  journal={Inverse Problems},
  volume={32},
  number={1},
  pages={015005},
  year={2016},
  publisher={IOP Publishing}
}

@article{klibanov2018numerical,
  title={A numerical method to solve a phaseless coefficient inverse problem from a single measurement of experimental data},
  author={Klibanov, Michael V and Koshev, Nikolay A and Nguyen, Dinh-Liem and Nguyen, Loc H and Brettin, Aaron and Astratov, Vasily N},
  journal={SIAM J. Imaging Sci.},
  volume={11},
  number={4},
  pages={2339--2367},
  year={2018},
  publisher={SIAM}
}

@article{klibanov2019coefficient,
  title={A coefficient inverse problem with a single measurement of phaseless scattering data},
  author={Klibanov, Michael V and Nguyen, Dinh-Liem and Nguyen, Loc H},
  journal={ SIAM J. Appl. Math.},
  volume={79},
  number={1},
  pages={1--27},
  year={2019},
  publisher={SIAM}
}

@article{ning2025direct,
  title={A direct sampling method and its integration with deep learning for inverse scattering problems with phaseless data},
  author={Ning, Jianfeng and Han, Fuqun and Zou, Jun},
  journal={SIAM J. Sci. Comput.},
  volume={47},
  number={2},
  pages={C343--C368},
  year={2025},
  publisher={SIAM}
}

@article{Gao2019MachineLB,
  title={Machine learning based data retrieval for inverse scattering problems with incomplete data},
  author={Yu Gao and Kai Zhang},
  journal={J. Inverse Ill-posed Probl.},
  year={2019},
  volume={29},
  pages={249 - 266},
  url={}
}

@article{Yin2020ANN,
  title={A neural network scheme for recovering scattering obstacles with limited phaseless far-field data},
  author={Weishi Yin and Wenhong Yang and Hongyu Liu},
  journal={J. Comput. Phys.},
  year={2020},
  volume={417},
  pages={109594},
  url={}
}

@article{Dong2018ARB,
  title={A reference ball based iterative algorithm for imaging acoustic obstacle from phaseless far-field data},
  author={Heping Dong and Deyue Zhang and Yukun Guo},
  journal={Inverse Probl. Imaging},
  volume = {13},
number = {1},
pages = {177-195},
year = {2019}
}

@article{Ji2019InverseAS,
  title={Inverse Acoustic Scattering with Phaseless Far Field Data: Uniqueness, Phase Retrieval, and Direct Sampling Methods},
  author={Xia Ji and Xiaodong Liu and Bo Zhang},
  journal={SIAM J. Imaging Sci.},
  year={2019},
  volume={12},
  pages={1163-1189}
}

@article{Ammari2013PartialDR,
  title={Partial Data Resolving Power of Conductivity Imaging from Boundary Measurements},
  author={Habib M. Ammari and Josselin Garnier and Knut S{\o}lna},
  journal={SIAM J. Math. Anal.},
  year={2013},
  volume={45},
  pages={1704-1722},
  url={}
}

@book{Colton1992InverseAA,
  title={Inverse acoustic and electromagnetic scattering theory},
  author={Colton, David L and Kress, Rainer},
  volume={93},
  year={1998},
  publisher={Springer}
}

@book{beilina2012approximate,
  title={Approximate global convergence and adaptivity for coefficient inverse problems},
  author={Beilina, Larisa and Klibanov, Michael Victor},
  year={2012},
  publisher={Springer Science \& Business Media}
}

@book{cakoni2022inverse,
  title={Inverse scattering theory and transmission eigenvalues},
  author={Cakoni, Fioralba and Colton, David and Haddar, Houssem},
  year={2022},
  publisher={SIAM}
}

@book{ito2014inverse,
  title={Inverse problems: Tikhonov theory and algorithms},
  author={Ito, Kazufumi and Jin, Bangti},
  volume={22},
  year={2014},
  publisher={World Scientific}
}

@article{submitted,
  title={A Spectral Model-Informed Neural Network for Inverse Source Problems},
  author={Nguyen, Dinh-Liem and Nguyen, Nhung H. and Ravi, Aravinth},
  journal={submitted},
  year={2026}
}

@article{Rzio2026ARO,
  title={A Review of Recent Machine Learning Approaches for Inverse Scattering Problems},
  author={Sofia R{\'e}zio and Pedro Guimar{\~a}es and Pedro Serranho},
  journal={Archives of Computational Methods in Engineering},
  year={2026},
  url={}
}

@article{Khoo2018SwitchNetAN,
  title={SwitchNet: a neural network model for forward and inverse scattering problems},
  author={Yuehaw Khoo and Lexing Ying},
  journal={SIAM J. Sci. Comput.},
  year={2018},
  volume={41},
  pages={A3182-A3201},
  url={}
}

@article{Fan2019SolvingEI,
title = {Solving electrical impedance tomography with deep learning},
author = {Yuwei Fan and Lexing Ying},
journal = {J. Comput. Phys.},
volume = {404},
pages = {109119},
year = {2020}
}

@article{Zhang2022SolvingTW,
  title={Solving the Wide-band Inverse Scattering Problem via Equivariant Neural Networks},
  author={Borong Zhang and Leonardo Zepeda-N'unez and Qin Li},
  journal={J of Comput Appl Math},
  year={2024},
  volume={451},
  url={}
}

@article{Yin2024PhysicsawareDL,
  title={Physics-aware deep learning framework for the limited aperture inverse obstacle scattering problem},
  author={Yunwen Yin and Liang Yan},
  journal={SIAM J Math Anal},
  year={2024},
  volume={56(3)},
  url={}
}

@article{Guo2021PhysicsED,
  title={Physics Embedded Deep Neural Network for Solving Full-Wave Inverse Scattering Problems},
  author={Rui Guo and Zhichao Lin and Tao Shan and Xiaoqian Song and Maokun Li and F. Yang and Shenheng Xu and Aria Abubakar},
  journal={IEEE Transactions on Antennas and Propagation},
  year={2021},
  volume={70},
  pages={6148-6159},
  url={}
}

@article{le2026inverse,
  title={Inverse scattering without phase: Carleman convexification and phase retrieval via the Wentzel--Kramers--Brillouin approximation},
  author={Le, Thuy T and Nguyen, Phuong M and Nguyen, Loc H},
  journal={Comput. Meth. Appl. Mech. Eng.},
  volume={448},
  pages={118439},
  year={2026},
  publisher={Elsevier}
}
